\documentclass[10pt,a4paper]{amsart}
\usepackage{fourier} 
\usepackage{ulem}

\usepackage{amsthm}
\usepackage{mathrsfs}
\usepackage{float}
\usepackage{amsmath}
\usepackage[all]{xy}\usepackage{xypic}
\usepackage{braket}
\usepackage{color}
\usepackage[colorlinks = true,citecolor=blue]{hyperref}
\hypersetup{
  pdftitle   = {},
  pdfauthor  = {},
  pdfcreator = {\LaTeX\ with package \flqq hyperref\frqq}
}

\DeclareMathAlphabet{\mathpzc}{OT1}{pzc}{m}{it}
\usepackage{tikz-cd}
\usepackage{tkz-graph}
\usetikzlibrary{shapes.geometric,arrows, shapes}

\newtheorem{theorem}{Theorem}[section]
\newtheorem*{theorem*}{Theorem}
\newtheorem{theorem-non}{Theorem}
\newtheorem{proposition}[theorem]{Proposition}

\newtheorem*{lemma*}{Lemma}
\newtheorem{corollary}[theorem]{Corollary}

\newtheorem*{conjecture*}{Conjecture}

\theoremstyle{definition}

\newtheorem{example}[theorem]{Example}

\theoremstyle{remark}
\newtheorem{remark}[theorem]{Remark}

\numberwithin{equation}{section}

\begin{document}
\title[Calabi-Yau metrics on canonical bundles of flag manifolds]{Calabi-Yau metrics on canonical bundles of complex flag manifolds}

%

\author{Eder M. Correa}
\address{IMPA, Estrada Dona Castorina 110, Rio de Janeiro, 22460-320, Brazil}
\email{edermoraes@impa.br}

\author{Lino Grama}
\address{IMECC-Unicamp, Departamento de Matem\'{a}tica. Rua S\'{e}rgio Buarque de Holanda,
651, Cidade Universit\'{a}ria Zeferino Vaz. 13083-859, Campinas - SP, Brazil}
\email{linograma@gmail.com}

\begin{abstract} 
In the present paper we provide a description of complete Calabi-Yau metrics on the canonical bundle of generalized complex flag manifolds. By means of Lie theory we give an explicit description of complete Ricci-flat K\"{a}hler metrics obtained through the Calabi ansatz technique. We use this approach to provide several explicit examples of noncompact complete Calabi-Yau manifolds, these examples include canonical bundles of non-toric flag manifolds (e.g. Grassmann manifolds and full flag manifolds).
\end{abstract}

\maketitle
\tableofcontents

\section{Introduction}




\label{S8.3}


In this paper we present a complete description of how to construct complete Ricci-flat K\"{a}hler metrics on the canonical line bundle of complex flag manifolds. The main tool which we apply in the context of complex flag manifolds is the Calabi ansatz technique \cite{CALABIANSATZ}, which provides a constructive method to obtain K\"{a}hler-Einstein metrics on total spaces of holomorphic vector bundles over K\"{a}hler-Einstein manifolds.  The basic idea in Calabi's construction is to use the Hermitian vector bundle structure over a K\"{a}hler-Einstein manifold to reduce the K\"{a}hler-Einstein condition, which is generally a Monge-Amp\`{e}re equation, to an ordinary differential equation \cite{KOBAYASHIANSATZ}.

A big family of well known examples of Calabi-Yau metrics provided by Calabi's technique are obtained in the setting of toric Fano manifolds (see for instance \cite{abreu}, \cite{KAWAI}). This family includes complex projective spaces, which are also examples of flag manifolds. 

Our main motivation for studying Calabi's construction in the setting of flag manifolds is the possibility of using the algebraic background which underlies this class of homogeneous spaces (semi-simple Lie theory, representation theory, etc.) to provide some explicitly description of geometric structures, in this case Calabi-Yau structures.

One of the most well known examples on which the construction of Ricci-flat K\"{a}hler metrics and scalar flat metrics has been widely studied is given by the total space $$\mathscr{O}(-k) \to \mathbb{C}{\rm{P}}^{n},$$ where $k > 0$, and on its disk bundle $\mathscr{D}(-k) \subset \mathscr{O}(-k)$ (see for instance \cite{abreu}, \cite{LEBRUN}, \cite{PEDERSEN}). Since $\mathbb{C}{\rm{P}}^{n}$ defines an example of toric manifold which is also a flag manifold, it deserves a special attention.

For the canonical bundle $K_{\mathbb{C}{\rm{P}}^{n}}=\mathscr{O}(-n-1)$ the toric geometry which underlies $\mathbb{C}{\rm{P}}^{n}$ allows us to describe $c_{1}(\mathbb{C}{\rm{P}}^{n})$ by means of combinatorial elements (Polytope facets \cite{abreu}, \cite{GKAHLER}). Therefore, we have a suitable description for all elements involved in Calabi's construction, namely, curvature and connection forms on $K_{\mathbb{C}{\rm{P}}^{n}}$, and K\"{a}hler-Einsten metric (Fubini-Study) on $\mathbb{C}{\rm{P}}^{n}$.

It is worth to point out that the same ideas briefly described above for $K_{\mathbb{C}{\rm{P}}^{n}}$ also can be applied in the general setting of canonical line bundles of toric K\"{a}hler-Einstein Fano manifolds \cite{KAWAI}.

In general complex flag manifolds are not toric manifolds, thus we do not have a combinatorial data to describe the geometric structures involved in Calabi's method. However, in this work we show how to overcome this problem by using elements of representation theory of simple Lie algebras. The key point in our approach is to employ elements of Lie theory to describe the local potential which defines the first Chern class $c_{1}(X_{P})$ for an arbitrary complex flag manifold $X_{P} = G^{\mathbb{C}}/P$. This description allows us to make explicit computations for Calabi-Yau metrics which can be constructed on the total space $K_{X_{P}} = \det(T^{\ast}X_{P})$ over $X_{P}$.

The main result which we establish in this work is the following (see Section \ref{sec-CY} for details):

\begin{theorem-non}

Let $(X_{P},\omega_{X_{P}})$ be a complex flag manifold associated to some parabolic subgroup $P = P_{\Theta} \subset G^{\mathbb{C}}$, such that $\dim_{\mathbb{C}}(X_{P}) = n$. Then, the total space $K_{X_{P}}$ admits a complete Ricci-flat K\"{a}hler metric with K\"{a}hler form given by
$$\omega_{CY} = (2\pi u + C)^{\frac{1}{n+1}} \pi^{\ast}\omega_{X_{P}} - \frac{\sqrt{-1}}{n+1}(2\pi u + C)^{-\frac{n}{n+1}}\nabla \xi \wedge \overline{\nabla \xi},
$$
where $C>0$ is some positive constant and $u \colon K_{X_{P}} \to \mathbb{R}_{\geq 0}$ is given by $u([g,\xi]) = |\xi|^{2}$, $\forall [g,\xi] \in K_{X_{P}}$. Furthermore, the above K\"{a}hler form is completely determined by the quasi-potential $\varphi \colon G^{\mathbb{C}} \to \mathbb{R}$ defined by
$$\varphi(g) = \displaystyle  \frac{1}{2\pi} \log \Big (\prod_{\alpha \in \Sigma \backslash \Theta}||gv_{\omega_{\alpha}}^{+}||^{2\langle \delta_{P},h_{\alpha}^{\vee} \rangle} \Big),$$
for every $g \in G^{\mathbb{C}}$. Therefore $(K_{X_{P}},\omega_{CY})$ is a (complete) noncompact Calabi-Yau manifold with Calabi-Yau metric $\omega_{CY}$ completely determined by $\Theta \subset \Sigma$, where $\Sigma$ denotes the set of simple roots of $\mathfrak{g}^\mathbb{C}= {\rm Lie }\, (G^{\mathbb{C}})$.
\end{theorem-non}

The idea of the proof of the above theorem is to combine the Calabi ansatz technique \cite{CALABIANSATZ} with the description of quasi-potentials for invariant K\"ahler metrics on flag manifolds provided by Azad and Biswas in \cite{AZAD}. By means of this idea we can describe explicit examples of complete Calabi-Yau metrics. In this way, besides recovering many well known results related to $X_{P} = \mathbb{C}{\rm{P}}^{n}$ (which are obtained from different methods), we also establish a new class of examples in the context of non-toric flag manifolds.

This paper is organized as follows: In Section 2, we introduce the basic notation and results concerned with Chern classes of holomorphic line bundles over flag manifolds, then after that we prove Theorem 1. In Section 3, as an application of Theorem 1 we provide explicit examples of complete Calabi-Yau metrics on the total space of the canonical bundle for a huge class of complex flag manifolds, including Grassmann manifolds, Wallach's flag manifolds, partial flag manifolds and full flag manifolds of classical Lie groups $\rm{SL}(n,\mathbb{C})$, $\rm{SO}(n,\mathbb{C})$ and $\rm{Sp}(2n,\mathbb{C})$.

In Section 4, we discuss some questions which arise when we try to apply Theorem 1 in the setting of complex flag manifolds associated to exceptional Lie groups. These manifolds include the complex Cayley plane and the Freudenthal variety.


\section{Calabi-Yau metrics on the canonical bundle of flag manifolds}
\setcounter{equation}{0}%
\label{S8.3Sub8.3.2}

\subsection{Preliminary results: line bundles over flag manifolds}

The results which we will cover in this subsection can be found in \cite{AZAD} and \cite[Appendix D.1]{EDERTHESIS}.

Let $\mathfrak{g}^{\mathbb{C}}$ be a complex simple Lie algebra, by fixing a Cartan subalgebra $\mathfrak{h}$ and a simple root system $\Sigma \subset \mathfrak{h}^{\ast}$, we have a decomposition of $\mathfrak{g}^{\mathbb{C}}$ given by
\begin{center}
$\mathfrak{g}^{\mathbb{C}} = \mathfrak{n}^{-} \oplus \mathfrak{h} \oplus \mathfrak{n}^{+}$, 
\end{center}
where $\mathfrak{n}^{-} = \sum_{\alpha \in \Pi^{-}}\mathfrak{g}_{\alpha}$ and $\mathfrak{n}^{+} = \sum_{\alpha \in \Pi^{+}}\mathfrak{g}_{\alpha}$, here we denote by $\Pi = \Pi^{+} \cup \Pi^{-}$ the root system associated to the simple root system $\Sigma = \{\alpha_{1},\ldots,\alpha_{l}\} \subset \mathfrak{h}^{\ast}$. We also denote by $\kappa$ the Cartan-Killing form of $\mathfrak{g}^{\mathbb{C}}$.

Now, given $\alpha \in \Pi^{+}$ we have $h_{\alpha} \in \mathfrak{h}$ such  that $\alpha = \kappa(\cdot,h_{\alpha})$, we can choose $x_{\alpha} \in \mathfrak{g}_{\alpha}$ and $y_{\alpha} \in \mathfrak{g}_{-\alpha}$ such that $[x_{\alpha},y_{\alpha}] = h_{\alpha}$. For every $\alpha \in \Sigma$ we can take 
$$h_{\alpha}^{\vee} = \frac{2}{\kappa(h_{\alpha},h_{\alpha})}h_{\alpha},$$ 
from this we have the fundamental weights $\{\omega_{\alpha} \ | \ \alpha \in \Sigma\} \subset \mathfrak{h}^{\ast}$, where $\omega_{\alpha}(h_{\beta}^{\vee}) = \delta_{\alpha \beta}$, $\forall \alpha, \beta \in \Sigma$. We denote by $\Lambda_{\mathbb{Z}_{\geq 0}}^{\ast} = \bigoplus_{\alpha \in \Sigma}\mathbb{Z}_{\geq 0}\omega_{\alpha}$ the set of integral dominant weights of $\mathfrak{g}^{\mathbb{C}}$.

From the representation theory of semi-simple Lie algebras, given $\mu \in \Lambda_{\mathbb{Z}_{\geq 0}}^{\ast}$ we have an irreducible $\mathfrak{g}^{\mathbb{C}}$-module $V(\mu)$ with highest weight $\mu$, we denote by $v_{\mu}^{+} \in V(\mu)$ the highest weight vector associated to $\mu \in  \Lambda_{\mathbb{Z}_{\geq 0}}^{\ast}$. 

Let $G^{\mathbb{C}}$ be a connected, simply connected and complex Lie group with simple Lie algebra $\mathfrak{g}^{\mathbb{C}}$, and consider $G \subset G^{\mathbb{C}}$ as being a compact real form for $G^{\mathbb{C}}$. Given a parabolic subgroup $P \subset G^{\mathbb{C}}$, without loss of generality, we can suppose

\begin{center}
$P  = P_{\Theta}$, \ for some \ $\Theta \subset \Sigma$.
\end{center}
By definition we have $P_{\Theta} = N_{G^{\mathbb{C}}}(\mathfrak{p}_{\Theta})$, where ${\text{Lie}}(P_{\Theta}) = \mathfrak{p}_{\Theta} \subset \mathfrak{g}^{\mathbb{C}}$ is given by

\begin{center}

$\mathfrak{p}_{\Theta} = \mathfrak{n}^{+} \oplus \mathfrak{h} \oplus \mathfrak{n}(\Theta)^{-}$, \ with \ $\mathfrak{n}(\Theta)^{-} = \displaystyle \sum_{\alpha \in \langle \Theta \rangle^{-}} \mathfrak{g}_{\alpha}$.

\end{center}
For our purpose it will be useful to consider the following basic subgroups

\begin{center}

$T^{\mathbb{C}} \subset B \subset P \subset G^{\mathbb{C}}$.

\end{center}

For each element in the above chain of subgroups we have the following characterization

\begin{itemize}

\item $T^{\mathbb{C}} = \exp(\mathfrak{h})$,  \ \ (complex torus)

\item $B = N^{+}T^{\mathbb{C}}$, where $N^{+} = \exp(\mathfrak{n}^{+})$, \ \ (Borel subgroup)

\item $P = P_{\Theta} = N_{G^{\mathbb{C}}}(\mathfrak{p}_{\Theta})$, for some $\Theta \subset \Sigma \subset \mathfrak{h}^{\ast}$. \ \ (parabolic subgroup)

\end{itemize}
Associated to the above data we will be concerned to study the {\it complex generalized flag manifold}  
$$
X_{P} = G^{\mathbb{C}} / P = G /G \cap P.
$$
The following Theorem allows us to describe all $G$-invariant K\"{a}hler structures on $X_{P}$
\begin{theorem}[Azad-Biswas,\cite{AZAD}]
\label{AZADBISWAS}
Let $\omega \in \Omega^{(1,1)}(X_{P})^{G}$ be a closed invariant real $(1,1)$-form, then we have

\begin{center}

$\pi^{\ast}\omega = \sqrt{-1}\partial \overline{\partial}\varphi$,

\end{center}
where $\pi \colon G^{\mathbb{C}} \to X_{P}$ and $\varphi \colon G^{\mathbb{C}} \to \mathbb{R}$ is given by 
\begin{center}
$\varphi(g) = \displaystyle \sum_{\alpha \in \Sigma \backslash \Theta}c_{\alpha}\log||gv_{\omega_{\alpha}}^{+}||$, 
\end{center}
with $c_{\alpha} \in \mathbb{R}_{\geq 0}$, $\forall \alpha \in \Sigma \backslash \Theta$. Conversely, every function $\varphi$ as above defines a closed invariant real $(1,1)$-form $\omega_{\varphi} \in \Omega^{(1,1)}(X_{P})^{G}$. Moreover, if $c_{\alpha} > 0$,  $\forall \alpha \in \Sigma \backslash \Theta$, then $\omega_{\varphi}$ defines a K\"{a}hler form on $X_{P}$.

\end{theorem}

\begin{remark}
It is worth to point out that the norm $|| \cdot ||$ in the above Theorem is a norm induced by a fixed $G$-invariant inner product on $V(\omega_{\alpha})$, $\forall \alpha \in \Sigma \backslash \Theta$. 
\end{remark}
Let $X_{P}$ be a flag manifold associated to some parabolic subgroup $P = P_{\Theta} \subset G^{\mathbb{C}}$. According to the above Theorem, by taking a fundamental weight $\omega_{\alpha} \in \Lambda_{\mathbb{Z}_{\geq0}}^{\ast}$ such that $\alpha \in \Sigma \backslash \Theta$, we can associate to this weight a closed real $G$-invariant $(1,1)$-form $\eta_{\alpha} \in \Omega^{(1,1)}(X_{P})^{G}$ which satisfies 
$$\pi^{\ast}\eta_{\alpha} = \sqrt{-1}\partial \overline{\partial} \varphi_{\omega_{\alpha}},$$
where $\pi \colon G^{\mathbb{C}} \to G^{\mathbb{C}} / P = X_{P}$ and $\varphi_{\omega_{\alpha}}(g) = \displaystyle \frac{1}{2\pi}\log||gv_{\omega_{\alpha}}^{+}||^{2}$.





Since each $\omega_{\alpha} \in \Lambda_{\mathbb{Z}_{\geq 0}}^{\ast}$ is an integral dominant weight, we can associate to it a holomorphic character $\chi_{\omega_{\alpha}} \colon T^{\mathbb{C}} \to \mathbb{C}^{\times}$, such that $(d\chi_{\omega_{\alpha}})_{e} = \omega_{\alpha}$, see for instance \cite[p. 466]{TAYLOR}. Given a parabolic subgroup $P \subset G^{\mathbb{C}}$, we can take an extension $\chi_{\omega_{\alpha}} \colon P \to \mathbb{C}^{\times}$ and define a holomorphic line bundle

\begin{equation}
\label{C8S8.2Sub8.2.1Eq8.2.4}
L_{\chi_{\omega_{\alpha}}} =  G^{\mathbb{C}} \times_{\chi_{\omega_{\alpha}}} \mathbb{C}_{-\omega_{\alpha}},
\end{equation}\\
as a vector bundle associated to the $P$-principal bundle $P \hookrightarrow G^{\mathbb{C}} \to G^{\mathbb{C}}/P$.
\begin{remark}
\label{remarkcocycle}
In the above description we consider $\mathbb{C}_{-\omega_{\alpha}}$ as a $P$-space with the action $pz = \chi_{\omega_{\alpha}}(p)^{-1}z$, $\forall p \in P$ and $\forall z \in \mathbb{C}$. Therefore, in terms of $\check{C}$ech cocycle, if we consider an open cover $X_{P} = \bigcup_{i \in I}U_{i}$ and
$G^{\mathbb{C}} = \{(U_{i})_{\alpha \in I}, \psi_{ij} \colon U_{i} \cap U_{j} \to P\}$, then we have 

$$L_{\chi_{\omega_{\alpha}}} = \Big \{(U_{i})_{i \in I},\chi_{\omega_{\alpha}}^{-1} \circ \psi_{i j} \colon U_{i} \cap U_{j} \to \mathbb{C}^{\times} \Big \},$$
thus $L_{\chi_{\omega_{\alpha}}} = \{g_{ij}\} \in \check{H}^{1}(X_{P},\mathscr{O}_{X_{P}}^{\ast})$, with $g_{ij} = \chi_{\omega_{\alpha}}^{-1} \circ \psi_{i j}$, where $i,j \in I$.
\end{remark}

The next results provide a complete description of holomorphic line bundles over complex flag manifolds
\begin{proposition}
\label{C8S8.2Sub8.2.3P8.2.7}
Let $X_{P}$ be a flag manifold associated to some parabolic subgroup $P = P_{\Theta}\subset G^{\mathbb{C}}$. Then for every fundamental weight $\omega_{\alpha} \in \Lambda_{\mathbb{Z}_{\geq 0}}^{\ast}$, such that $\alpha \in \Sigma \backslash \Theta$, we have
\begin{equation}
\label{C8S8.2Sub8.2.3Eq8.2.28}
c_{1}(L_{\chi_{\omega_{\alpha}}}) = [\eta_{\alpha}].
\end{equation}
%
%
\end{proposition}

\begin{proposition}
\label{C8S8.2Sub8.2.3P8.2.6}
Let $X_{P}$ be a flag manifold associated to some parabolic subgroup $P = P_{\Theta}\subset G^{\mathbb{C}}$. Then we have

\begin{center}

${\text{Pic}}(X_{P}) = H^{1,1}(X_{P},\mathbb{Z}) = H^{2}(X_{P},\mathbb{Z}) = \displaystyle \bigoplus_{\alpha \in \Sigma \backslash \Theta}\mathbb{Z}[\eta_{\alpha}]$.

\end{center}

\end{proposition}

An important holomorphic line bundle to be considered over $X_{P}$ in our approach is its anticanonical bundle

\begin{center}
$-K_{X_{P}} = \det \big(T^{(1,0)}X_{P} \big) = \bigwedge^{(n,0)}(T^{(1,0)}X_{P} \big)$,
\end{center}
here we suppose $\dim_{\mathbb{C}}(X_{P}) = n$. 

In the context of complex flag manifolds the anticanonical line bundle can be described as follows: consider the identification

\begin{center}

$\mathfrak{m} = \displaystyle \sum_{\alpha \in \Pi^{+} \backslash \langle \Theta \rangle^{+}} \mathfrak{g}_{-\alpha} = T_{x_{0}}^{(1,0)}X_{P}$,
\end{center}
where $x_{0} = eP \in X_{P}$. We have the following characterization for $T^{(1,0)}X_{P}$


\begin{center}

$T^{(1,0)}X_{P} = G^{\mathbb{C}} \times_{P} \mathfrak{m}$,

\end{center}
here we observe that the twisted product on the right side above is obtained from the isotropy representation ${\rm{Ad}} \colon P \to {\rm{GL}}(\mathfrak{m})$ as an associated holomorphic vector bundle. 

 

%
%
%
%
%


Let us define $\delta_{P} \in \mathfrak{h}^{\ast}$ by

\begin{center}

$\delta_{P} = \displaystyle \sum_{\alpha \in \Pi^{+} \backslash \langle \Theta \rangle^{+}} \alpha$,

\end{center}
a straightforward computation shows that 

\begin{center}

$\det \circ {\text{Ad}} = \chi_{\delta_{P}}^{-1}$, 

\end{center}
from this we have the following result


\begin{proposition}
\label{C8S8.2Sub8.2.3Eq8.2.35}
Let $X_{P}$ be a flag manifold associated to some parabolic subgroup $P = P_{\Theta}\subset G^{\mathbb{C}}$, then we have

$$-K_{X_{P}} = \det \big(T^{(1,0)}X_{P} \big) =\det \big ( G^{\mathbb{C}} \times_{P} \mathfrak{m} \big )= L_{\chi_{\delta_{P}}}.$$

\end{proposition}
The above result allows us to write 

$$- K_{X_{P}} = \big \{(U_{i})_{i \in I},\chi_{\delta_{P}}^{-1} \circ \psi_{i j} \colon U_{i} \cap U_{j} \to \mathbb{C}^{\times} \big \},$$
see Remark \ref{remarkcocycle}. Moreover, since the holomorphic character associated to $\delta_{P}$ can be written as
$$\chi_{\delta_{P}} = \displaystyle \prod_{\alpha \in \Sigma \backslash \Theta} \chi_{\omega_{\alpha}}^{\langle \delta_{P},h_{\alpha}^{\vee} \rangle},$$
we have the following characterization

$$-K_{X_{P}} =  L_{\chi_{\delta_{P}}} = \displaystyle \bigotimes_{\alpha \in \Sigma \backslash \Theta} L_{\chi_{\omega_{\alpha}}}^{\otimes \langle \delta_{P},h_{\alpha}^{\vee} \rangle}.$$

From Proposition \ref{C8S8.2Sub8.2.3P8.2.6} and Proposition \ref{C8S8.2Sub8.2.3Eq8.2.35} we have the following corollary:
\begin{corollary}
\label{COROLLARY}
The Chern class of $-K_{X_{P}}$ is given by
$$c_{1}(-K_{X_{P}}) = \displaystyle \sum_{\alpha \in \Sigma \backslash \Theta} \big \langle \delta_{P},h_{\alpha}^{\vee} \big \rangle c_1(L_{\chi_{\omega_{\alpha}}}).$$
\end{corollary}

%
%
%
%
%
%
%
%
%
%

\begin{remark}
In order to do some local computations it will be convenient for us to consider the open set defined by the ``opposite" big cell in $X_{P}$. This open set is a distinguished coordinate neighbourhood $U \subset X_{P}$ of $x_{0} = eP \in X_{P}$ defined by the maximal Schubert cell. A brief description for the opposite big cell can be done as follows. Let $\Pi = \Pi^{+} \cup \Pi^{+}$ be the root system associated to the simple root system $\Sigma \subset \mathfrak{h}^{\ast}$, from this we can define the opposite big cell $U \subset X_{P}$ by

\begin{center}

 $U =  B^{-}x_{0} = R_{u}(P_{\Theta})^{-}x_{0} \subset X_{P}$,  

\end{center}
 where $B^{-} = \exp(\mathfrak{h} \oplus \mathfrak{n}^{-})$ and
 
 \begin{center}
 
 $R_{u}(P_{\Theta})^{-} = \displaystyle \prod_{\alpha \in \Pi^{-} \backslash \langle \Theta \rangle^{-}}N_{\alpha}^{-}$, \ \ (opposite unipotent radical)
 
 \end{center}
with $N_{\alpha}^{-} = \exp(\mathfrak{g}_{\alpha})$, $\forall \alpha \in \Pi^{-} \backslash \langle \Theta \rangle^{-}$. The opposite big cell defines a contractible open dense subset of $X_{P}$. For further results about Schubert cells and Schubert varieties we suggest \cite{MONOMIAL}.

\end{remark}

\subsection{Calabi-Yau metrics on  $K_{X_{P}}$} \label{sec-CY}

It is well known that every generalized flag manifold admits a $G$-invariant K\"{a}hler form.  The K\"{a}hler potential was described by  Azad and Biswas, and is an important element in our approach. The next result is a consequence of Theorem \ref{AZADBISWAS} and corollary \ref{COROLLARY}.

\begin{proposition}
\label{Prop2.8}
 Let $X_{P}$ be a complex flag manifold associated to some parabolic subgroup $ P = P_{\Theta} \subset G^{\mathbb{C}}$, for some $\Theta \subset \Sigma$. Then there exists a $G$-invariant K\"{a}hler form $\omega_{X_{P}} \in c_{1}(X_{P})$ completely determined by the (quasi-potential) $\varphi \colon G^{\mathbb{C}} \to \mathbb{R}$ defined by
$$
\varphi(g) = \displaystyle  \frac{1}{2\pi} \log \Big (\prod_{\alpha \in \Sigma \backslash \Theta}||gv_{\omega_{\alpha}}^{+}||^{2\langle \delta_{P},h_{\alpha}^{\vee} \rangle} \Big),
$$
where $v_{\omega_{\alpha}}^{+} $ is the highest weight vector associated to the fundamental weight $\omega_\alpha$, for each $\alpha \in \Sigma \backslash \Theta$. 
\end{proposition}  

\begin{remark}
By ``completely determined" we mean that for every locally defined smooth section $s_{U} \colon U \subset X_{P} \to G^{\mathbb{C}}$, we have 
$$
\omega_{X} |_{U} = \sqrt{-1}\partial \overline{\partial}(s_{U}^{\ast}(\varphi)),
$$
see \cite{AZAD} for details.
\end{remark}

By means of the above local description we can take a Hermitian structure $H$ on $K_{X_{P}}$ defined as follows. Consider an open cover $X_{P} = \bigcup_{i \in I} U_{i}$ which trivializes both $P \hookrightarrow G^{\mathbb{C}} \to X_{P}$ and $K_{X_{P}} \to X_{P}$. Now, take a collection of local sections $(s_{i})_{i \in I}$, such that $s_{i} \colon U_{i} \to G^{\mathbb{C}}$. From these we define $q_{i} \colon U_{i} \to \mathbb{R}_{+}$ by
\begin{equation}
\label{functionshermitian}
q_{i} =  {\mathrm{e}}^{2\pi \varphi \circ s_{i}} = \prod_{\alpha \in \Sigma \backslash \Theta}||s_{i}v_{\omega_{\alpha}}^{+}||^{2\langle \delta_{P},h_{\alpha}^{\vee} \rangle},
\end{equation}
for every $i \in I$. These functions $(q_{i})_{i \in I}$ satisfy $q_{j} = |\chi_{\delta_{P}} \circ \psi_{ij}|^{2}q_{i}$ on $U_{i} \cap U_{j} \neq \emptyset$, here we have used that $s_{j} = s_{i}\psi_{ij}$ on $U_{i} \cap U_{j} \neq \emptyset$, and $pv_{\omega_{\alpha}}^{+} = \chi_{\omega_{\alpha}}(p)v_{\omega_{\alpha}}^{+}$, for every $p \in P$ and $\alpha \in \Sigma \backslash \Theta$. From the above collection of smooth functions we can define a Hermitian structure $H$ on $K_{X_{P}}$ by taking on each trivialization $\phi_{i} \colon K_{X_{P}} \to U_{i} \times \mathbb{C}$ a metric defined by
\begin{equation}
\label{hermitian}
H((x,v),(x,w)) = q_{i}(x) v\overline{w},
\end{equation}
for $(x,v),(x,w) \in K_{X_{P}}|_{U_{i}} \cong U_{i} \times \mathbb{C}$. The above Hermitian structure induces a Chern connection $\nabla = d + \partial \log H$ with curvature $F_{\nabla}$ satisfying 
$$
F_{\nabla} = 2\pi \sqrt{-1}\omega_{X_{P}}.
$$
Notice that $c_{1}(K_{X_{P}}) = -[\omega_{X_{P}}] = -c_{1}(X_{P})$, thus from Theorem \ref{AZADBISWAS} we have that $(X_{P},\omega_{X_{P}})$ is a K\"{a}hler-Einstein Fano manifold with ${\text{Ric}}(\omega_{X_{P}}) = 2\pi \omega_{X_{P}}$, e.g. \cite[Appendix D.1]{EDERTHESIS}.

The following result due to E. Calabi ensures the existence of Ricci-flat K\"{a}hler metrics on the canonical line bundle of generalized complex flag manifolds. For details we suggest \cite{CALABIANSATZ} (see also \cite[p. 108, Theorem 8.1]{SALAMONANSATZ}).

\begin{theorem}[E. Calabi]
\label{CALABI3}
 Let $(X,\omega_{X})$ be a compact K\"{a}hler-Einstein manifold such that $c_{1}(X) > 0$, i.e. a K\"{a}hler-Einstein Fano manifold. Then, there exists a complete Ricci-flat metric on the manifold defined by the total space $K_{X} = \det(T^{\ast}X)$. 
\end{theorem}

\begin{remark}
\label{REMARKIMPORTANT}
Let us give an outline of the proof of Theorem \ref{CALABI3}, the details of the ideas which we describe below can be found in \cite[Appendix D.2]{EDERTHESIS}.

Under the hypotheses of Theorem \ref{CALABI3}, consider $\omega_{X}$ as being a K\"{a}hler-Einstein metric on $X$ which satisfies  ${\text{Ric}}(\omega_{X}) = t\omega_{X}$, $t>0$, and take a Chern connection $(A_{U})_{U \in \mathscr{U}}$ on $K_{X}$ such that $\sqrt{-1}dA_{U} = -t\omega_{X}$. From this, we can define $\omega_{0} \in \Omega^{(1,1)}(K_{X})$ by

\begin{equation}
\label{Ansatzform}
\omega_{0} = f(u)\pi^{\ast}\omega_{X} - \displaystyle \frac{\sqrt{-1}}{t}f'(u)\nabla \xi \wedge \overline{\nabla \xi}, 
\end{equation}
where $f \colon \mathbb{R} \to \mathbb{R}$ is a smooth function, $u([x,\xi]) = |\xi|^{2}$, $\forall [x,\xi] \in K_{X}$, and $\nabla \xi \wedge \overline{\nabla \xi}$ is a $(1,1)$-form obtained by gluing $\nabla \xi|_{U} \wedge \overline{\nabla \xi|_{U}}$, where

\begin{center}

$\nabla \xi|_{U} = d\xi_{U} + \xi_{U} A_{U},$

\end{center}
is defined in local coordinates $(z_{U},\xi_{U}) \in K_{X}|_{U}$, $U \in \mathscr{U}$.

Now, let us briefly describe the main steps in the proof of Theorem \ref{CALABI3}:

\begin{enumerate}

\item If $f > 0$ and $f' > 0 \implies \omega_{0}$ is a positive $(1,1)$-form. Furthermore, It is straightforward to check that (locally) $\omega_{0} = d\eta$, where

\begin{center}

$\eta = \displaystyle -\frac{\sqrt{-1}}{t} \Big (f A_{U} - \frac{f d\overline{\xi}_{U}}{\overline{\xi}_{U}} \Big ),$

\end{center}
thus we have that $\omega_{0}$ defines a K\"{a}hler form;

\item Consider $\theta \in \Omega^{(n,0)}(K_{X})$, such that

\begin{center}
$\theta_{\nu}(X_{1},\ldots,X_{n}) = (\pi^{\ast}\nu)(X_{1},\ldots,X_{n})$,
\end{center}
$\forall \nu \in K_{X}$, $X_{1},\ldots,X_{n} \in T_{\nu}^{(1,0)}K_{X}$. By setting $\Omega = d\theta \in \Omega^{(n+1,0)}(K_{X})$, we obtain
\begin{center}

$d\Omega = 0 \implies \overline{\partial}\Omega = 0,$

\end{center}
notice that locally we have $\theta = \xi_{U}\sigma_{U}$ and $\Omega = d\xi_{U} \wedge \sigma_{U}$, where $\sigma_{U}$ is a nonvanishing holomorphic section of $K_{X}|_{U}$. Thus, it follows that  $\Omega = d\theta$ defines a holomorphic volume form;

\item By considering the norm $||\cdot||_{\omega_{0}}$ on $\bigwedge^{(n+1,0)}(K_{X})$ induced by $\omega_{0}$, we have 

\begin{center}
$(n+1)!\Omega \wedge \overline{\Omega} = \iota^{(n+1)^{2}}||\Omega||_{\omega_{0}}^{2}\omega_{0}^{(n+1)},$
\end{center}
from the equation above we obtain the following condition
\begin{center}
$||\Omega||_{\omega_{0}} = {\text{const.}} \iff f(u)^{n}f'(u) = {\text{const.}}$;
\end{center}

\item Hence, if $f(u) = (tu + C)^{\frac{1}{n+1}} $, $C>0$ $\implies ||\Omega||_{\omega_{0}}^{2} = (n+1)^{2}$;

\item Moreover, $||\Omega||_{\omega_{0}} =  (n+1)^{2} \implies \nabla^{\omega_{0}}\Omega = 0 \implies {\text{Hol}}^{0}(\nabla^{\omega_{0}}) \subset {\rm{SU}}(n+1)$.

\end{enumerate}
From the above facts we have that

\begin{center}

$\omega_{0} = \displaystyle (tu+C)^{\frac{1}{n+1}} \pi^{\ast} \omega_{X} - \frac{\sqrt{-1}}{n+1}(tu+C)^{-\frac{n}{n+1}}\nabla \xi \wedge \overline{\nabla \xi},$

\end{center}
defines a Ricci-flat K\"{a}hler metric on $K_{X}$.

Now, consider the Riemannian metric $g_{0} = \omega_{0}({\text{id}} \otimes J)$ on $K_{X}$, we have

\begin{center}
\label{C8S8.3Sub8.3.1Eq8.3.14}
$g_{0} = \displaystyle (tu + C)^{\frac{1}{n+1}} \pi^{\ast}g_{X} + \frac{1}{n+1}(tu + C)^{-\frac{n}{n+1}}{\text{Re}} \big (\nabla \xi \otimes \overline{\nabla \xi} \big ).$
\end{center}
If we take a divergent $C^{1}$-path $\gamma \colon [0,a) \to (K_{X},g_{0})$, we have essentially two possibilities:

\begin{enumerate}

\item $\pi \circ \gamma \colon [0,a) \to (X,g_{X})$ is a divergent $C^{1}$-path (horizontal divergence);

\item $\pi \circ \gamma \colon [0,a) \to (X,g_{X})$ is not a divergent $C^{1}$-path (vertical divergence).

\end{enumerate}
In the first case above the completeness of $(X,g_{X})$ shows us that $\gamma \colon [0,a) \to (K_{X},g_{0})$ has infinite length. In the second case the divergence of the integral

\begin{center}

$\displaystyle \sqrt{\frac{1}{n+1}} \int_{0}^{+ \infty} \frac{1}{(tu + C)^{\frac{n}{2(n+1)}}} d(u^{\frac{1}{2}})$, 

\end{center}
implies that the curve $\gamma \colon [0,a) \to (K_{X},g_{0})$ has infinite length on the vertical radial direction, i.e., if $\gamma$ is a divergent curve on the vertical radial direction it has infinite length. Therefore, from the characterization of complete Riemannian manifolds provided by \cite[p. 179, Proposition 1]{DIVPATH}, it follows that $g_{0}$ is a complete Riemannian metric.

\end{remark}

Now we are able to prove our main result. 

\begin{theorem}
\label{C8S8.3Sub8.3.2Teo8.3.5}
Let $(X_{P},\omega_{X_{P}})$ be a complex flag manifold associated to some parabolic subgroup $P = P_{\Theta} \subset G^{\mathbb{C}}$, such that $\dim_{\mathbb{C}}(X_{P}) = n$. Then, the total space $K_{X_{P}}$ admits a complete Ricci-flat K\"{a}hler metric with K\"{a}hler form given by

\begin{equation}
\label{C8S8.3Sub8.3.2Eq8.3.15}
\omega_{CY} = (2\pi u + C)^{\frac{1}{n+1}} \pi^{\ast}\omega_{X_{P}} - \frac{\sqrt{-1}}{n+1}(2\pi u + C)^{-\frac{n}{n+1}}\nabla \xi \wedge \overline{\nabla \xi},
\end{equation}\\
where $C>0$ is some positive constant and $u \colon K_{X_{P}} \to \mathbb{R}_{\geq 0}$ is given by $u([g,\xi]) = |\xi|^{2}$, $\forall [g,\xi] \in K_{X_{P}}$. Furthermore, the above K\"{a}hler form is completely determined by the quasi-potential $\varphi \colon G^{\mathbb{C}} \to \mathbb{R}$ defined by
$$\varphi(g) = \displaystyle  \frac{1}{2\pi} \log \Big (\prod_{\alpha \in \Sigma \backslash \Theta}||gv_{\omega_{\alpha}}^{+}||^{2\langle \delta_{P},h_{\alpha}^{\vee} \rangle} \Big),$$
for every $g \in G^{\mathbb{C}}$. Therefore, $(K_{X_{P}},\omega_{CY})$ is a (complete) noncompact Calabi-Yau manifold with Calabi-Yau metric $\omega_{CY}$ completely determined by $\Theta \subset \Sigma$.
\end{theorem}

\begin{proof}
The fact that $(K_{X_{P}},\omega_{CY})$ is a complete Ricci-flat K\"{a}hler manifold follows from Theorem \ref{CALABI3}, see also Remark \ref{REMARKIMPORTANT}. Therefore, we just need to verify that ${\text{Hol}}(\nabla^{\omega_{CY}}) \subseteq {\rm{SU}}(n+1)$, where $\nabla^{\omega_{CY}}$ denotes the Levi-Civita connection induced on $\bigwedge^{(n+1,0)}(K_{X_{P}})$ by the K\"{a}hler metric $\omega_{CY}$, and check that $\omega_{CY}$ is completely determined by $\varphi$.

In order to see that ${\text{Hol}}(\nabla^{\omega_{CY}}) \subseteq {\rm{SU}}(n+1)$ we observe that since $K_{X_{P}} \to X_{P}$
is a vector bundle, it follows that $K_{X_{P}}$ has the same homotopy type of $X_{P}$. Once $X_{P}$ is simply connected, we have that $K_{X_{P}}$ is also simply connected. Hence, we have 
$${\text{Hol}}(\nabla^{\omega_{CY}}) = {\text{Hol}}^{0}(\nabla^{\omega_{CY}}) \subseteq {\rm{SU}}(n+1),$$
see \cite[p. 27, Proposition 2.2.6]{JOYCE} for the last argument.

Now, we will check that $\omega_{CY}$ is completely determined by $\varphi$. At first, we notice that on the trivializing open set given by the opposite big cell $U = R_{u}(P_{\Theta})^{-}x_{0} \subset X_{P}$ we have holomorphic coordinates $(nP,\xi_{U})$ on $\pi^{-1}(U) = K_{X_{P}}|_{U}$, where $n \in R_{u}(P_{\Theta})^{-}$ and $\xi_{U} \colon \pi^{-1}(U) \to \mathbb{C}$. From this we can take a local section $s_{U} \colon U \subset X_{P} \to G^{\mathbb{C}}$ defined by $s_{U}(nP) = n \in G^{\mathbb{C}}$, $\forall nP \in U$.

Since on the open set $U \subset X_{P}$ we have $\omega_{X_{P}} |_{U} = \sqrt{-1}\partial \overline{\partial}(s_{U}^{\ast}(\varphi))$, it follows that the $(1,1)$-horizontal component of $\omega_{CY}$ is (locally) defined by
$$
\omega_{X_{P}} |_{U} =  \displaystyle \frac{\sqrt{-1}}{2\pi} \partial \overline{\partial}\log \Big (\prod_{\alpha \in \Sigma \backslash \Theta}||nv_{\omega_{\alpha}}^{+}||^{2\langle \delta_{P},h_{\alpha}^{\vee} \rangle} \Big).
$$

From Remark \ref{REMARKIMPORTANT} we have $\nabla \xi_{U} = d\xi_{U} + \xi_{U} \pi^{\ast}A_{U}$, for some Chern connection locally described by $\nabla |_{U} = d + A_{U}$. We claim that on the $(1,1)$-vertical component of $\omega_{CY}$ we have
$$\nabla \xi_{U} = \displaystyle d\xi_{U} +  \xi_{U}\partial \log \Big (\prod_{\alpha \in \Sigma \backslash \Theta}||nv_{\omega_{\alpha}}^{+}||^{2\langle \delta_{P},h_{\alpha}^{\vee} \rangle} \Big).$$

In fact, if we take the Hermitian structure on $K_{X_{P}}$ defined by the collection of functions $(q_{i})_{i \in I}$ as in \ref{functionshermitian}, namely
$$
H((x,v), (x,w)) = {\mathrm{e}}^{2\pi \varphi(s_{i}(x))}v\overline{w},
$$
see \ref{hermitian}. On coordinates  $(nP,\xi_{U})$ in $\pi^{-1}(U) = K_{X_{P}}|_{U}$ we have $H = {\mathrm{e}}^{2\pi \varphi(n)}\xi_{U} \overline{\xi_{U}}$. Therefore, if we take a local section $\sigma_{U} \colon U \subset X_{P} \to K_{X_{P}}$, defined by $\sigma_{U}(nP) = (s_{U}(nP),1) = (n,1) \in K_{X_{P}}|_{U}$, we have $H(\sigma_{U},\sigma_{U}) =  {\mathrm{e}}^{2\pi \varphi \circ s_{U}}$, thus we obtain
$$
\nabla|_{U} = d + \partial \log H(\sigma_{U},\sigma_{U})  = d + 2 \pi \partial (s_{U}^{\ast}(\varphi)).
$$
It follows that $\nabla \xi_{U} = d\xi_{U} +  2 \pi \xi_{U}\partial (s_{U}^{\ast}(\varphi))$. Hence, the K\"{a}hler form $\omega_{CY}$ is (locally) determined by the forms

\begin{itemize}

\item $\omega_{X_{P}} |_{U} =  \displaystyle \frac{\sqrt{-1}}{2\pi} \partial \overline{\partial}\log \Big (\prod_{\alpha \in \Sigma \backslash \Theta}||nv_{\omega_{\alpha}}^{+}||^{2\langle \delta_{P},h_{\alpha}^{\vee} \rangle} \Big)$, (Horizontal) 

\medskip

\item $\nabla \xi_{U} = \displaystyle d\xi_{U} +  \xi_{U}\partial \log \Big (\prod_{\alpha \in \Sigma \backslash \Theta}||nv_{\omega_{\alpha}}^{+}||^{2\langle \delta_{P},h_{\alpha}^{\vee} \rangle} \Big)$. \ \ (Vertical) 

\end{itemize}
Thus, the Ricci-flat structure defined by $\omega_{CY}$ can be completely described by means of the quasi-potential $\varphi \colon G^{\mathbb{C}} \to \mathbb{R}$, from these we have the desired result.\end{proof}

One of the most important feature of the above Theorem is that it allows us to assign to each subset $\Theta \subset \Sigma$ a complete noncompact Calabi-Yau manifold for which we have the Calabi-Yau metric completely determined by elements of Lie theory.

\begin{remark}
We observe that Calabi-Yau metrics on canonical line bundles of homogeneous spaces also were studied in \cite{HKN} by using different methods.
\end{remark}

\section{Examples of complete noncompact Calabi-Yau manifolds via Lie theory}

\subsection{Warm-up examples: complex projective spaces}

Now we will apply Theorem \ref{C8S8.3Sub8.3.2Teo8.3.5} in a more concrete situation in order to provide a huge class of examples of noncompact Calabi-Yau manifolds obtained via Calabi's technique. Despite Calabi-Yau structures on the canonical bundle of complex projective spaces are well known (for instance, via toric geometry) we provide a different approach via Lie theory. We start by describing the construction on the building block of the general setting.

\begin{remark}
Unless otherwise stated, in the examples which we will describe below we will use the conventions of \cite{SMA} for the realization of classical simple Lie algebras as matrix Lie algebras.
\end{remark}

\begin{example}[Calabi-Yau metrics on $\mathscr{O}(-2)$]
Consider $G^{\mathbb{C}} = {\rm{SL}}(2,\mathbb{C})$, we fix a triangular decomposition for $\mathfrak{sl}(2,\mathbb{C})$ given by

\begin{center}

$\mathfrak{sl}(2,\mathbb{C}) = \Big \langle x = \begin{pmatrix}
 0 & 1 \\
 0 & 0
\end{pmatrix} \Big \rangle_{\mathbb{C}} \oplus  \Big \langle h = \begin{pmatrix}
 1 & 0 \\
 0 & -1
\end{pmatrix} \Big \rangle_{\mathbb{C}} \oplus \Big \langle y = \begin{pmatrix}
 0 & 0 \\
 1 & 0
\end{pmatrix} \Big \rangle_{\mathbb{C}}$.

\end{center}
All the information about the above decomposition is codified in $\Sigma = \{\alpha\}$ and $\Pi = \{\alpha,-\alpha\}$, our set of integral dominant weights in this case is given by

\begin{center}

$\Lambda_{\mathbb{Z}_{\geq 0}}^{\ast} = \mathbb{Z}_{\geq 0}\omega_{\alpha} $.

\end{center} 
We take $P = B$ (Borel subgroup) and from this we obtain
\begin{center}
$X_{B} = {\rm{SL}}(2,\mathbb{C})/B = \mathbb{C}{\rm{P}}^{1}$.
\end{center}
Notice that since $V(\omega_{\alpha}) = \mathbb{C}^{2}$ and $v_{\omega_{\alpha}}^{+} = e_{1}$, it follows that over the opposite big cell $U = N^{-}x_{0} \subset X_{B}$ we have

\begin{center}

$\omega_{\mathbb{C}{\rm{P}}^{1}} = \displaystyle  \frac{\sqrt{-1}}{2\pi} \langle \delta_{B},h_{\alpha}^{\vee} \rangle\partial \overline{\partial} \log \Bigg ( \Big |\Big |\begin{pmatrix}
 1 & 0 \\
 z & 1
\end{pmatrix} e_{1} \Big| \Big|^{2} \Bigg ) = \frac{\sqrt{-1}}{\pi} \partial \overline{\partial} \log (1+|z|^{2}).$

\end{center}
Furthermore, we have $-K_{\mathbb{C}{\rm{P}}^{1}} =  T^{(1,0)}\mathbb{C}{\rm{P}}^{1} = T\mathbb{C}{\rm{P}}^{1}$, and 
$K_{\mathbb{C}{\rm{P}}^{1}} = T^{\ast}\mathbb{C}{\rm{P}}^{1}$. From the previous Theorem we can equip  $K_{\mathbb{C}{\rm{P}}^{1}} = \mathscr{O}(-2)$ with a complete Ricci-Flat metric induced by the K\"{a}hler form
$$\omega_{CY} = (2\pi u + C)^{\frac{1}{2}} \pi^{\ast}\omega_{\mathbb{C}{\rm{P}}^{1}} - \frac{\sqrt{-1}}{2}(2\pi u + C)^{-\frac{1}{2}}\nabla \xi \wedge \overline{\nabla \xi}.$$
If we take the trivializing open set $U =  N^{-}x_{0} \subset \mathbb{C}{\rm{P}}^{1}$ and consider local coordinates $(z,\xi)$ on $\mathscr{O}(-2)|_{U}$, we have the following local expression for $\omega_{CY}$:

\begin{equation}
\label{C8S8.3Sub8.3.2Eq8.3.16}
\omega_{CY} = \displaystyle \frac{\sqrt{2\pi|\xi|^{2}+C}}{\pi(1+|z|^{2})^{2}}\sqrt{-1}dz \wedge d\overline{z} -  \frac{\sqrt{-1}}{2\sqrt{2\pi|\xi|^{2}+C}} \Big (d\xi +\frac{\xi \overline{z}dz}{1+|z|^{2}} \Big) \wedge \Big ( d\overline{\xi} +\frac{\overline{\xi} zd\overline{z}}{1+|z|^{2}} \Big).
\end{equation}\\
Since this example is quite simple, let us briefly verify the Ricci-flatness condition. Consider the tautological form $\theta \in \Omega^{(1,0)}(\mathscr{O}(-2))$. If we take a nonvanishing local holomorphic section section $\sigma \colon U \subset \mathbb{C}{\rm{P}}^{1} \to \mathscr{O}(-2)$, the local expression of $\Omega = d\theta \in \Omega^{(2,0)}(\mathscr{O}(-2))$ is given by $\Omega = d\xi \wedge \sigma$, notice that locally $\theta = \xi \sigma$.

Now, since 
$$2! \Omega \wedge \overline{\Omega} = (\sqrt{-1})^{2^{2}}||\Omega||_{\omega_{CY}}^{2}\omega_{CY}^{2} = ||\Omega||_{\omega_{CY}}^{2}\omega_{CY}^{2},$$
from equation \ref{C8S8.3Sub8.3.2Eq8.3.16} we have $\omega_{CY}^{2} = \frac{-\sqrt{-1}}{2}\omega_{\mathbb{C}{\rm{P}}^{1}} \wedge d\xi \wedge d\overline{\xi}$, thus we obtain
$$2! \Omega \wedge \overline{\Omega} = \frac{-\sqrt{-1}}{2}||\Omega||_{\omega_{CY}}^{2}\omega_{\mathbb{C}{\rm{P}}^{1}} \wedge d\xi \wedge d\overline{\xi}.$$
On the other hand we have 
$$\Omega \wedge \overline{\Omega}  = \big ( d\xi \wedge \sigma \big ) \wedge \big ( d\overline{\xi} \wedge \overline{\sigma} \big) = -\sqrt{-1}||\sigma||_{\omega_{\mathbb{C}{\rm{P}}^{1}}}^{2}\omega_{\mathbb{C}{\rm{P}}^{1}} \wedge d\xi \wedge d\overline{\xi},$$
here we have used $\sigma \wedge \overline{\sigma} = \sqrt{-1}||\sigma||_{\omega_{\mathbb{C}{\rm{P}}^{1}}}^{2}\omega_{\mathbb{C}{\rm{P}}^{1}}$. Without loss of generality we can suppose $\sigma$ unitary instead holomorphic, notice that the previous computation remains the same. Therefore, since $||\sigma||_{\omega_{\mathbb{C}{\rm{P}}^{1}}}^{2} = 1$, we obtain
$$||\Omega||_{\omega_{CY}}^{2} = 4 \implies \nabla^{\omega_{CY}}\Omega = 0.$$
Thus, ${\text{Hol}}(\nabla^{\omega_{CY}}) \subset {\rm{SU}}(2)$, it follows that $(\mathscr{O}(-2),\omega_{CY},\Omega)$ is a noncompact Calabi-Yau manifold. \qed\\

It is worthwhile to point out that besides of rich geometric features of the metric described above, the construction of Ricci-flat metrics on $\mathscr{O}(-2)$ (Eguchi-Hanson metric) is also interesting for its applications in Mathematical Physics, see for instance \cite{EGUCHI}. Moreover, this manifold also defines an example of toric Calabi-Yau surface which has many applications in mirror symmetry, see for example \cite{TORICYAU}. We also point out that since ${\rm{SU}}(2) = {\rm{Sp}}(1)$ the metric described above is also a hyperk\"{a}hler metric.

\end{example}
Our next examples will be an illustration of how to compute Calabi's metric directly from Theorem \ref{C8S8.3Sub8.3.2Teo8.3.5}. 

\begin{example}[Calabi-Yau metrics on $\mathscr{O}(-3)$]\label{ex-cp2}
Consider now the case $G^{\mathbb{C}} = {\rm{SL}}(3,\mathbb{C})$. We fix the Cartan subalgebra $\mathfrak{h} \subset \mathfrak{sl}(3,\mathbb{C})$ given by the diagonal matrices whose the trace is equal zero. In this case the set of simple roots is given by 

$$\Sigma = \Big \{ \alpha_{1} = \epsilon_{1}-\epsilon_{2}, \alpha_{2} = \epsilon_{2} - \epsilon_{3}\Big \},$$
where $\epsilon_{l} \colon {\text{diag}}\{a_{1},\ldots,a_{3} \} \mapsto a_{l}$, $ \forall l = 1, \ldots, 3$, and the corresponding set of positive roots is $\Pi^{+} = \{\alpha_{1},\alpha_{2},\alpha_{1}+\alpha_{2}\}$. By taking $P = P_{\{\alpha_{2}\}}$ we have
$$X_{P} = \mathbb{C}{\rm{P}}^{2},$$
and the quasi-potential $\varphi \colon {\rm{SL}}(3,\mathbb{C}) \to \mathbb{R}$ is
$$\varphi(g) = \displaystyle \frac{\langle \delta_{P}, h_{\alpha_{1}}^{\vee}\rangle}{2\pi} \log ||g v_{\omega_{\alpha_{1}}}^{+}||^{2}.$$
Since $V(\omega_{\alpha_{1}}) = \mathbb{C}^{3}$ and $v_{\omega_{\alpha_{1}}}^{+} = e_{1}$, on the open set defined by the opposite big cell $U = B^{-}x_{0} \subset X_{P}$ we have $\omega_{\mathbb{C}{\rm{P}}^{2}} \in c_{1}(\mathbb{C}{\rm{P}}^{2})$ given by
$$
\omega_{\mathbb{C}{\rm{P}}^{2}} =  \displaystyle \frac{\langle \delta_{P}, h_{\alpha_{1}}^{\vee}\rangle}{2\pi} \sqrt{-1}\partial \overline{\partial } \log \Big( 1 + |z_{1}|^{2} + |z_{2}|^{2}\Big ).$$
Here we have used the parameterization of $U = B^{-}x_{0} \subset X_{P}$ in complex coordinates given by the matrices 
$$n = \begin{pmatrix}
1 & 0 & 0 \\
z_{1} & 1 & 0 \\                  
z_{2}  & 0 & 1
 \end{pmatrix},$$
with $z_{1}, z_{2} \in \mathbb{C}$, notice that the above coordinate system is obtained directly from the exponential map. It is straightforward to verify that $\langle \delta_{P}, h_{\alpha_{1}}^{\vee}\rangle = 3$, thus by applying the result of Theorem \ref{C8S8.3Sub8.3.2Teo8.3.5} we obtain the following (local) expression for the K\"{a}hler form associated to the  Ricci-flat metric on $K_{\mathbb{C}{\rm{P}}^{2}} = \mathscr{O}(-3)$ 
$$\omega_{CY} = (2\pi |\xi|^{2} + C)^{\frac{1}{3}} \omega_{\mathbb{C}{\rm{P}}^{2}} - \frac{\sqrt{-1}}{3}(2\pi |\xi|^{2}+ C)^{-\frac{2}{3}}\nabla \xi \wedge \overline{\nabla \xi},$$
for some constant $C>0$. Here we consider the local coordinates $(nP,\xi) \in \mathscr{O}(-3)|_{U}$, thus we have for the vertical and horizontal components of the above form the following local description
\begin{equation}
\label{C8S8.3Sub8.3.2Eq8.3.17}
\omega_{\mathbb{C}{\rm{P}}^{2}} = \displaystyle \frac{3\sqrt{-1}}{2 \pi} \partial \overline{\partial } \log \Big( 1 + \sum_{k=1}^{2}|z_{k}|^{2} \Big ) \ \ and \ \ \nabla \xi = d \xi + 3\xi \partial \log \Big( 1 + \sum_{k=1}^{2}|z_{k}|^{2}\Big ).
\end{equation}\\
Hence, $(\mathscr{O}(-3),\omega_{CY})$ defines an example of noncompact complete Calabi-Yau manifold.

As in the previous example the Calabi-Yau manifold obtained from the holomorphic line bundle $\mathscr{O}(-3) \to \mathbb{C}{\rm{P}}^{2}$ is a very interesting example of noncompact Calabi-Yau threefold which has many applications in Mathematical Physics, more precisely in String Theory, see for instance \cite[p. 429-431]{TOPVERTEX} and \cite{STRINGYAU}. \qed
\end{example}

The next example is a brief generalization of the description which we did above for $X_{P} = \mathbb{C}{\rm{P}}^{2}$.

\begin{example}[Calabi-Yau metrics on $\mathscr{O}(-n-1)$] \label{CY-CPn}
Let us briefly describe the application of Theorem \ref{C8S8.3Sub8.3.2Teo8.3.5} for $X_{P} = \mathbb{C}{\rm{P}}^{n}$. At first we recall some basic data related to the Lie algebra $\mathfrak{sl}(n+1,\mathbb{C})$. By fixing the Cartan subalgebra $\mathfrak{h} \subset \mathfrak{sl}(n+1,\mathbb{C})$ given by diagonal matrices whose the trace is equal zero, we have the set of simple roots given by
$$\Sigma = \Big \{ \alpha_{l} = \epsilon_{l} - \epsilon_{l+1} \ \Big | \ l = 1, \ldots,n\Big\},$$
here $\epsilon_{l} \colon {\text{diag}}\{a_{1},\ldots,a_{n+1} \} \mapsto a_{l}$, $ \forall l = 1, \ldots,n+1$. Therefore the set of positive roots is given by
$$\Pi^+ = \Big \{ \alpha = \epsilon_{i} - \epsilon_{j} \ \Big | \ i<j  \Big\}. $$

In this example we consider $\Theta = \Sigma \backslash \{\alpha_{1}\}$ and $P = P_\Theta$. Now we take the open set defined by the opposite big cell $U =  R_{u}(P_{\Theta})^{-}x_{0} \subset \mathbb{C}{\rm{P}}^{n}$, where $x_0=eP$ (trivial coset) and  
\begin{center}
 $ R_{u}(P_{\Theta})^{-} = \displaystyle \prod_{\alpha \in \Pi^{-} \backslash \langle \Theta \rangle^{-}}N_{\alpha}^{-}$, \ with \ $N_{\alpha}^{-} = \exp(\mathfrak{g}_{\alpha})$, $\forall \alpha \in \Pi^{-} \backslash \langle \Theta \rangle^{-}$.
 \end{center}
We remark that in this case the open set is parameterized by matrices $n \in B^{-}$ of the form
\begin{center}
$n = \begin{pmatrix}
1 & 0 &\cdots & 0 \\
z_{1} & 1  &\cdots & 0 \\                  
\ \vdots  & \vdots &\ddots  & \vdots  \\
z_{n} & 0 & \cdots &1 
 \end{pmatrix}$,
\end{center}
with $(z_{1},\ldots,z_{n}) \in \mathbb{C}^{n}$. Notice that the above coordinate system is induced directly from the exponential map $\exp \colon {\text{Lie}}(R_{u}(P)^{-}) \to R_{u}(P)^{-}$.

From this we can take a local section $s_{U} \colon U \to {\rm{SL}}(n+1,\mathbb{C})$, such that $s_{U}(nx_{0}) = n \in {\rm{SL}}(n+1,\mathbb{C})$. The expression of $\omega_{\mathbb{C}{\rm{P}}^{n}} \in c_{1}(\mathbb{C}{\rm{P}}^{n})$ over the opposite big cell is given by
$$\omega_{\mathbb{C}{\rm{P}}^{n}} = \displaystyle \frac{(n+1)}{2\pi} \sqrt{-1}\partial \overline{ \partial} \log \Big (1 + \sum_{l = 1}^{n}|z_{l}|^{2} \Big ).$$
Thus from Theorem \ref{C8S8.3Sub8.3.2Teo8.3.5} we have the following local description for the Calabi-Yau metric on $K_{\mathbb{C}{\rm{P}}^{n}} = \mathscr{O}(-n-1)$:
$$\omega_{CY} = \big (2\pi |\xi|^{2} + C \big )^{\frac{1}{n+1}} \omega_{\mathbb{C}{\rm{P}}^{n}} - \frac{\sqrt{-1}}{n+1} \big (2\pi |\xi|^{2} + C \big)^{-\frac{n}{n+1}}\nabla \xi \wedge \overline{\nabla \xi},$$
for some constant $C>0$, where $\omega_{\mathbb{C}{\rm{P}}^{n}}$ is locally described as above, and $\nabla \xi$ is locally given by
\begin{equation}
\nabla \xi = d\xi + (n+1)\xi\partial \log \Big (1 + \sum_{l = 1}^{n}|z_{l}|^{2} \Big ).
\end{equation}
Therefore, $(\mathscr{O}(-n-1),\omega_{CY})$ defines a complete $(n+1)$-dimensional noncompact Calabi-Yau manifold. 


\end{example}

\begin{remark}
An important fact about the above well known examples is that the complex flag manifold $\mathbb{C}{\rm{P}}^{n}$ is also a toric manifold. Since the K\"{a}hler structure of toric manifolds are completely determined by combinatorial elements, see for example \cite{GKAHLER}, in some sense it makes the application of Calabi's technique somewhat manageable. Our point is that it is also possible to consider the underlying Lie theoretical data of $\mathbb{C}{\rm{P}}^{n}$ in order to compute the Calabi-Yau structure on its canonical line bundle $\mathscr{O}(-n-1)$. It is worth to point out that the construction of Calabi-Yau metrics by means of Calabi's technique on $K_{\mathbb{C}{\rm{P}}^{n}}$ was in fact the first example introduced in \cite{CALABIANSATZ}.
\end{remark}

\subsection{Generalized flag manifolds}

The result which we established in Theorem \ref{C8S8.3Sub8.3.2Teo8.3.5} allows us to describe explicitly a huge class of complete Calabi-Yau metrics beyond the toric setting. In what follows we describe some results which can be obtained from Theorem \ref{C8S8.3Sub8.3.2Teo8.3.5}.

\subsubsection{Wallach flag manifold}

The Wallach space $W_{6}$ is the homogeneous space defined by

\begin{center}

$W_{6} = {\rm{SU}}(3)/T^{2}$.

\end{center}
The manifold above also can be written as $W_{6} = {\rm{SL}}(3,\mathbb{C})/B$, where $B$ is a Borel subgroup of ${\rm{SL}}(3,\mathbb{C})$. This manifold is non-toric and appear in Wallach's classification of homogeneous spaces admitting metric of positive sectional curvature.  We have the following result:

\begin{proposition}
The total space of the canonical bundle $K_{W_{6}}$ over the Wallach space $W_{6} = {\rm{SL}}(3,\mathbb{C})/B$ admits a Calabi-Yau metric $\omega_{CY}$ (locally) defined by

\begin{center}

$\omega_{CY} = (2\pi |\xi|^{2} + C)^{\frac{1}{4}} \omega_{X_{B}} - \frac{\sqrt{-1}}{4}(2\pi |\xi|^{2} + C)^{-\frac{3}{4}}\nabla \xi \wedge \overline{\nabla \xi}$,

\end{center}
for some positive constant $C>0$, such that
\begin{equation}
\omega_{W_{6}} = \displaystyle \frac{\sqrt{-1}}{\pi} \Bigg [\partial \overline{\partial } \log \Big( 1 + \sum_{k=1}^{2}|z_{k}|^{2}\Big ) + \partial \overline{\partial } \log \Big( 1 + |z_{3}|^{2} + |z_{1}z_{3}-z_{2}|^{2}\Big ) \Bigg ],
\end{equation}\\
and 
\begin{equation}
\nabla \xi = d\xi + 2\xi \Bigg [\partial \log \Big( 1 + \sum_{k=1}^{2}|z_{k}|^{2}\Big ) + \partial \log \Big( 1 + |z_{3}|^{2} + |z_{1}z_{3}-z_{2}|^{2}\Big ) \Bigg ].
\end{equation}
\end{proposition}
\begin{proof} Since the metric
\begin{center}
$\omega_{CY} = (2\pi |\xi|^{2} + C)^{\frac{1}{4}} \omega_{X_{B}} - \frac{\sqrt{-1}}{4}(2\pi |\xi|^{2} + C)^{-\frac{3}{4}}\nabla \xi \wedge \overline{\nabla \xi}$,
\end{center}
is the Ricci-flat metric provided by Calabi's technique, we can apply Theorem \ref{C8S8.3Sub8.3.2Teo8.3.5} to describe its local expression.

Here again we consider for $G^{\mathbb{C}} = {\rm{SL}}(3,\mathbb{C})$ the same choice of Cartan subalgebra and conventions for the simple root system as in example \ref{ex-cp2}. Therefore the simple root system can be described as follows

\begin{center}

$\Sigma = \Big \{ \alpha_{1} = \epsilon_{1}-\epsilon_{2}, \alpha_{2} = \epsilon_{2} - \epsilon_{3}\Big \}$,

\end{center}
and the set of positive roots in this case is given by $\Pi^{+} = \{\alpha_{1},\alpha_{2},\alpha_{1}+\alpha_{2}\}$. Now we fix the canonical basis $\{e_{j}\}$ for $\mathbb{C}^{3}$ and the basis $\{e_{i} \wedge e_{j}\}_{i<j}$ for $\bigwedge^{2}(\mathbb{C}^{3})$.

We first observe that for ${W_{6}} = {\rm{SL}}(3,\mathbb{C})/B$ the quasi-potential $\varphi \colon {\rm{SL}}(3,\mathbb{C}) \to \mathbb{R}$ is given by
$$\varphi(g) =  \displaystyle \frac{\langle \delta_{B}, h_{\alpha_{1}}^{\vee}\rangle}{2\pi} \log ||g v_{\omega_{\alpha_{1}}}^{+}||^{2} +  \displaystyle \frac{\langle \delta_{B}, h_{\alpha_{2}}^{\vee}\rangle}{2\pi} \log ||g v_{\omega_{\alpha_{2}}}^{+}||^{2},$$
where $V(\omega_{\alpha_{1}}) =  \mathbb{C}^{3}$ and $V(\omega_{\alpha_{2}}) = \bigwedge^{2}(\mathbb{C}^{3})$, and the highest-weight vectors are, respectively, given by
\begin{center}
$ v_{\omega_{\alpha_{1}}}^{+} = e_{1}$ \ \ and \ \ $v_{\omega_{\alpha_{2}}}^{+} = e_{1} \wedge e_{2}$.
\end{center}

In order to compute the local expression of $\omega_{W_{6}} \in c_{1}({W_{6}})$, we take the open set defined by the opposite big cell $U = B^{-}x_{0} \subset {W_{6}}$, on which we have coordinates

\begin{center}

$n = \begin{pmatrix}
1 & 0 & 0 \\
z_{1} & 1 & 0 \\                  
z_{2}  & z_{3} & 1
 \end{pmatrix}$,

\end{center}
with $z_{1},z_{2},z_{3} \in \mathbb{C}$. 

\begin{remark}
\label{coordinates}
Notice that in this case the coordinates are also obtained from the exponential map. However, the coordinates $z_{1},z_{2},z_{3} \in \mathbb{C}$ are defined by algebraically independent polynomials, since the number of polynomials and the number of coordinates are the same we will not specify the polynomials.
\end{remark}

By applying Proposition \ref{Prop2.8} for the local section $s_{U} \colon U \to {\rm{SL}}(3,\mathbb{C})$ defined by $s_{U}(nB) = n$, $\forall nB \in U$, we obtain

\begin{center}

$\omega_{W_{6}} = \displaystyle \frac{\langle \delta_{B}, h_{\alpha_{1}}^{\vee}\rangle}{2\pi} \sqrt{-1}\partial \overline{\partial } \log \Big( 1 + |z_{1}|^{2} + |z_{2}|^{2}\Big ) + \frac{\langle \delta_{B}, h_{\alpha_{2}}^{\vee}\rangle}{2\pi} \sqrt{-1}\partial \overline{\partial } \log \Big( 1 + |z_{3}|^{2} + |z_{1}z_{3}-z_{2}|^{2}\Big )$.

\end{center}
From the Cartan matrix of $\mathfrak{sl}(3,\mathbb{C})$ we have $\langle \delta_{B}, h_{\alpha_{1}}^{\vee}\rangle = \langle \delta_{B}, h_{\alpha_{2}}^{\vee}\rangle = 2$. Notice that in this case we have $ \delta_{B} = 2\alpha_{1} + 2\alpha_{2}$. Therefore, from Theorem \ref{C8S8.3Sub8.3.2Teo8.3.5} we obtain a Calabi-Yau metric on $K_{W_{6}}$ (locally) given by

\begin{center}

$\omega_{CY} = (2\pi |\xi|^{2} + C)^{\frac{1}{4}} \omega_{X_{B}} - \frac{\sqrt{-1}}{4}(2\pi |\xi|^{2} + C)^{-\frac{3}{4}}\nabla \xi \wedge \overline{\nabla \xi}$,

\end{center}
for some $C>0$. We take coordinates $(n,\xi) \in K_{W_{6}}|_{U}$ in order to obtain the following expressions
\begin{equation}
\omega_{X_{B}} = \displaystyle \frac{\sqrt{-1}}{\pi} \Bigg [\partial \overline{\partial } \log \Big( 1 + \sum_{k=1}^{2}|z_{k}|^{2}\Big ) + \partial \overline{\partial } \log \Big( 1 + |z_{3}|^{2} + |z_{1}z_{3}-z_{2}|^{2}\Big ) \Bigg ],
\end{equation}\\
and 
\begin{equation}
\nabla \xi = d\xi + 2\xi \Bigg [\partial \log \Big( 1 + \sum_{k=1}^{2}|z_{k}|^{2}\Big ) + \partial \log \Big( 1 + |z_{3}|^{2} + |z_{1}z_{3}-z_{2}|^{2}\Big ) \Bigg ].
\end{equation}\\
From these we have a noncompact complete Calabi-Yau manifold $(K_{X_{B}},\omega_{CY})$ with complex dimension $4$.\end{proof}


\subsubsection{Complex Grassmannians}

The next result can be seen as a prototype for the application of Theorem \ref{C8S8.3Sub8.3.2Teo8.3.5} on K\"{a}hler manifolds defined by complex Grassmannians.

\begin{proposition}
The total space of the canonical bundle $K_{{\rm{Gr}}(2,\mathbb{C}^{4})}$ over the complex Grassmannian ${\rm{Gr}}(2,\mathbb{C}^{4})$ admits a complete Calabi-Yau metric $\omega_{CY}$ (locally) described by 

\begin{center}
$\omega_{CY} = (2\pi |\xi|^{2} + C)^{\frac{1}{5}} \omega_{{\rm{Gr}}(2,\mathbb{C}^{4})}  - \frac{\sqrt{-1}}{5}(2\pi |\xi|^{2} + C)^{-\frac{4}{5}}\nabla \xi \wedge \overline{\nabla \xi}$,
\end{center}
for some positive constant $C>0$, such that

\begin{equation}
\omega_{{\rm{Gr}}(2,\mathbb{C}^{4})} =  \displaystyle \frac{2\sqrt{-1}}{\pi}  \partial \overline{\partial} \log \Big (1+ \sum_{k = 1}^{4}|z_{k}|^{2} + \bigg |\det \begin{pmatrix}
 z_{1} & z_{3} \\
 z_{2} & z_{4}
\end{pmatrix} \bigg |^{2} \Big),
\end{equation}\\
 and 
 \begin{equation}
\nabla \xi = d\xi + 4 \xi \partial \log \Big (1+ \sum_{k = 1}^{4}|z_{k}|^{2} + \bigg |\det \begin{pmatrix}
 z_{1} & z_{3} \\
 z_{2} & z_{4}
\end{pmatrix} \bigg |^{2} \Big),
\end{equation}

\end{proposition}

\begin{proof} Consider $G^{\mathbb{C}} = {\rm{SL}}(4,\mathbb{C})$, here we use the same choice of Cartan subalgebra and conventions for the simple root system as in the previous examples. Since our simple root system is given by
$$\Sigma = \Big \{ \alpha_{1} = \epsilon_{1}-\epsilon_{2}, \alpha_{2} = \epsilon_{2} - \epsilon_{3}, \alpha_{3} = \epsilon_{3} - \epsilon_{4}\Big \},$$
by taking $\Theta = \Sigma \backslash \{\alpha_{2}\}$ we obtain  for $P = P_{\Theta}$ the flag manifold $X_{P} = {\rm{Gr}}(2,\mathbb{C}^{4})$. Notice that in this case we have
$${\text{Pic}}({\rm{Gr}}(2,\mathbb{C}^{4})) = \mathbb{Z}[\eta_{\alpha_{2}}].$$
Thus from Proposition \ref{C8S8.2Sub8.2.3Eq8.2.35} it follows that
$$-K_{{\rm{Gr}}(2,\mathbb{C}^{4})} = L_{\chi_{\omega_{\alpha_{2}}}}^{\otimes \langle \delta_{P},h_{\alpha_{2}}^{\vee} \rangle}.$$

By considering our Lie-theoretical conventions, we have
$$\Pi^{+} \backslash \langle \Theta \rangle^{+} = \Big \{ \alpha_{2}, \alpha_{1}+\alpha_{2},\alpha_{2}+\alpha_{3},\alpha_{1} + \alpha_{2} + \alpha_{3}\Big \},$$
hence

$$\delta_{P} = \displaystyle \sum_{\alpha \in \Pi^{+} \backslash \langle \Theta \rangle^{+}} \alpha = 2\alpha_{1} + 4 \alpha_{2} + 2\alpha_{3}.$$
By means of the Cartan matrix of  $\mathfrak{sl}(4,\mathbb{C})$ we can compute 
$$ \langle \delta_{P},h_{\alpha_{2}}^{\vee} \rangle = 4 \implies -K_{{\rm{Gr}}(2,\mathbb{C}^{4})} = L_{\chi_{\omega_{\alpha_{2}}}}^{\otimes4}.$$  
In what follows we will use the following notation: 
$$L_{\chi_{\omega_{\alpha_{2}}}}^{\otimes k} := \mathscr{O}_{\alpha_{2}}(k),$$
for every $k \in \mathbb{Z}$, therefore we have $K_{{\rm{Gr}}(2,\mathbb{C}^{4})} = \mathscr{O}_{\alpha_{2}}(-4)$. In order to compute the local expression of $\omega_{{\rm{Gr}}(2,\mathbb{C}^{4})} \in c_{1}(\mathscr{O}_{\alpha_{2}}(-4))$, we observe that in this case the quasi-potential $\varphi \colon  {\rm{SL}}(4,\mathbb{C}) \to \mathbb{R}$ is given by

\begin{center}

$\varphi(g) = \displaystyle \frac{\langle \delta_{P}, h_{\alpha_{2}}^{\vee}\rangle}{2\pi} \log ||g v_{\omega_{\alpha_{2}}}^{+}||^{2} = \displaystyle \frac{2}{\pi} \log ||g v_{\omega_{\alpha_{2}}}^{+}||^{2}$,

\end{center}
where $V(\omega_{\alpha_{2}}) = \bigwedge^{2}(\mathbb{C}^{4})$ and $v_{\omega_{\alpha_{2}}}^{+} =  e_{1} \wedge e_{2}$. We fix the basis $\{e_{i} \wedge e_{j}\}_{i<j}$ for $V(\omega_{\alpha_{2}}) = \bigwedge^{2}(\mathbb{C}^{4})$. Similarly to the previous examples we consider the open set defined by the opposite big cell $U = B^{-}x_{0} \subset {\rm{Gr}}(2,\mathbb{C}^{4})$. In this case we have the local coordinates $nx_{0} \in U$ given by

\begin{center}

$n = \begin{pmatrix}
1 & 0 & 0 & 0\\
0 & 1 & 0 & 0 \\                  
z_{1}  & z_{3} & 1 & 0 \\
z_{2}  & z_{4} & 0 & 1
 \end{pmatrix}$,

\end{center}
 with $z_{i} \in \mathbb{C}$, $i  =1,2,3,4$. Notice that the above coordinates are obtained directly from the exponential map $\exp \colon {\text{Lie}}(R_{u}(P)^{-}) \to R_{u}(P)^{-}$. From this we obtain

\begin{center}

$\varphi(n) = \displaystyle \frac{2}{\pi} \log \Big ( 1 + \sum_{k = 1}^{4}|z_{k}|^{2} + \bigg |\det \begin{pmatrix}
 z_{1} & z_{3} \\
 z_{2} & z_{4}
\end{pmatrix} \bigg |^{2} \Big)$,

\end{center}
and the following local expression for  $\omega_{{\rm{Gr}}(2,\mathbb{C}^{4})} \in c_{1}(\mathscr{O}_{\alpha_{2}}(-4))$
\begin{equation}
\label{C8S8.3Sub8.3.2Eq8.3.21}
\omega_{{\rm{Gr}}(2,\mathbb{C}^{4})} =  \displaystyle \frac{2 \sqrt{-1}}{\pi}  \partial \overline{\partial} \log \Big (1+ \sum_{k = 1}^{4}|z_{k}|^{2} + \big |\det \begin{pmatrix}
 z_{1} & z_{3} \\
 z_{2} & z_{4}
\end{pmatrix} \big |^{2} \Big).
\end{equation}\\
Now we can apply Theorem \ref{C8S8.3Sub8.3.2Teo8.3.5} in order to get a Ricci-flat metric on $\mathscr{O}_{\alpha_{2}}(-4)$ which is locally described by

\begin{center}
$\omega_{CY} = (2\pi |\xi|^{2} + C)^{\frac{1}{5}} \omega_{{\rm{Gr}}(2,\mathbb{C}^{4})}  - \frac{\sqrt{-1}}{5}(2\pi |\xi|^{2} + C)^{-\frac{4}{5}}\nabla \xi \wedge \overline{\nabla \xi}$,
\end{center}
for some constant $C>0$, where $\omega_{{\rm{Gr}}(2,\mathbb{C}^{4})}$ can be written as above in \ref{C8S8.3Sub8.3.2Eq8.3.21} and $\nabla \xi$ is locally given by
\begin{equation}
\label{C8S8.3Sub8.3.2Eq8.3.22}
\nabla \xi = d\xi + 4 \xi \partial \log \Big (1+ \sum_{k = 1}^{4}|z_{k}|^{2} + \bigg |\det \begin{pmatrix}
 z_{1} & z_{3} \\
 z_{2} & z_{4}
\end{pmatrix} \bigg |^{2} \Big).
\end{equation}\\
Thus $(\mathscr{O}_{\alpha_{2}}(-4),\omega_{CY})$ defines a complete noncompact Calabi-Yau manifold of complex dimension $5$. 
\end{proof}

The ideas of the above example can be easily extended to any complex Grassmannian ${\rm{Gr}}(k,\mathbb{C}^{n})$. Actually, by fixing the usual Lie theoretical data for the Lie algebra $ \mathfrak{sl}(n,\mathbb{C})$, we can describe the complete Ricci-flat metric on the total space of the canonical bundle
$$K_{{\rm{Gr}}(k,\mathbb{C}^{n})} \to {\rm{Gr}}(k,\mathbb{C}^{n}),$$
via Lie-theoretical objects like fundamental representations and simple root system. Notice that in the general case of complex Grassmannians we have ${\rm{Gr}}(k,\mathbb{C}^{n}) = {\rm{SL}}(n,\mathbb{C})/P_{\Sigma \backslash \{\alpha_{k}\}}$.

\subsubsection{Full flag manifold of ${\rm{SL}}(n+1,\mathbb{C})$}

Now we describe a more general result concerned with full flag manifolds 

\begin{theorem}

Consider $G^{\mathbb{C}} = {\rm{SL}}(n+1,\mathbb{C})$ and $B  \subset G^{\mathbb{C}} = {\rm{SL}}(n+1,\mathbb{C})$ (Borel subgroup). Then the total space of the canonical bundle $K_{X_{B}}$ over the complex full flag manifold $X_{B} = {\rm{SL}}(n+1,\mathbb{C})/B$ admits a complete Calabi-Yau metric $\omega_{CY}$ (locally) described by 

\begin{center}

$\omega_{CY} = (2\pi |\xi|^{2} + C)^{\frac{2}{n(n+1)+2}} \omega_{X_{B}} - \frac{2 \sqrt{-1}}{n(n+1)+2}(2\pi |\xi|^{2} + C)^{-\frac{n(n+1)}{n(n+1)+2}}\nabla \xi \wedge \overline{\nabla \xi},$

\end{center}
for some positive constant $C>0$, such that

\begin{itemize}

    \item $\omega_{X_{B}} = \displaystyle  \sum_{k = 1}^{n}\frac{\langle \delta_{B},h_{\alpha_{k}}^{\vee} \rangle}{2\pi} \sqrt{-1}\partial \overline{\partial}\log \Bigg (1 + \sum_{I \neq I_{0,k}} \big|\Delta_{I}^{(k)}(n^{-}(z)) \big|^{2}\Bigg)$; \ ( Horizontal )
    
    \item $\nabla \xi = d\xi + \displaystyle  \sum_{k = 1}^{n}\langle \delta_{B},h_{\alpha_{k}}^{\vee} \rangle \xi \partial \log \Bigg (1 + \sum_{I \neq I_{0,k}} \big|\Delta_{I}^{(k)}(n^{-}(z)) \big|^{2}\Bigg)$. \ (Vertical)
    
\end{itemize}

\end{theorem}
\begin{proof}
Consider $G^{\mathbb{C}} = {\rm{SL}}(n+1,\mathbb{C})$ and fix the same Lie-theoretical data as in example \ref{CY-CPn}. By taking $\Theta =  \emptyset$ we have the complex flag manifold $X_{B} = {\rm{SL}}(n+1,\mathbb{C})/B$, where $B \subset {\rm{SL}}(n+1,\mathbb{C})$ is the standard Borel subgroup. By Calabi's technique and our Theorem \ref{C8S8.3Sub8.3.2Teo8.3.5} it follows that the metric
$$\omega_{CY} = (2\pi |\xi|^{2} + C)^{\frac{2}{n(n+1)+2}} \omega_{X_{B}} - \frac{2 \sqrt{-1}}{n(n+1)+2}(2\pi |\xi|^{2} + C)^{-\frac{n(n+1)}{n(n+1)+2}}\nabla \xi \wedge \overline{\nabla \xi}$$
defines a Ricci-flat metric on $K_{X_{B}}$. In order to verify the second assertion let us introduce some notations. Let $U = N^{-}B \subset X_{B}$ be the opposite big cell, where $N^{-} = \exp(\mathfrak{n}^{-})$. This open set is parameterized by the holomorphic coordinates 
$$n = \begin{pmatrix}
  1 & 0 & 0 & \cdots & 0 \\
  z_{21} & 1 & 0 & \cdots & 0  \\
  z_{31} & z_{32} & 1 &\cdots & 0 \\
  \vdots & \vdots & \vdots & \ddots & \vdots \\
  z_{n+1,1} & z_{n+1,2} & z_{n+1,3} & \cdots & 1 
 \end{pmatrix},$$
where $ n = n^{-}(z) \in N^{-}$ and $z = (z_{ij}) \in \mathbb{C}^{\frac{n(n+1)}{2}}$. Notice that the above parameterization is induced from the exponential map, however each coordinate is represented by a polynomial, we will not specify the polynomials expressions by the same reason explained in Remark \ref{coordinates}.

Given $g \in {\rm{SL}}(n+1,\mathbb{C})$ we denote by
$$g = \begin{pmatrix}
  g_{11} & g_{12} & g_{13} & \cdots & g_{1,n+1} \\
  g_{21} & g_{22} & g_{23} & \cdots & g_{2,n+1}  \\
  g_{31} & g_{32} & g_{33}&\cdots & g_{3,n+1} \\
  \vdots & \vdots & \vdots & \ddots & \vdots \\
  g_{n+1,1} & g_{n+1,2} & g_{n+1,3} & \cdots & g_{n+1,n+1} 
 \end{pmatrix}.$$
 
We define for each subset $I = \{i_{1} < \cdots < i_{k}\} \subset \{1,\cdots,n+1\}$, with $1 \leq k \leq n$, the following polynomial functions $\Delta_{I}^{(k)} \colon {\rm{SL}}(n+1,\mathbb{C}) \to \mathbb{C}$, such that 

$$\Delta_{I}^{(k)}(g) = \det \begin{pmatrix}
  g_{i_{1}1} & g_{i_{1}2} & \cdots & g_{i_{1} k} \\
  g_{i_{2}1} & g_{i_{2}2} & \cdots & g_{i_{2} k}  \\
  \vdots & \vdots& \ddots & \vdots \\
  g_{i_{k}1} & g_{i_{k}2} & \cdots & g_{i_{k}k} 
 \end{pmatrix},$$
and we set $I_{0,k} = \{1,2,\ldots,k\}$. We have 
\begin{equation}
\label{C8S8.3Sub8.3.2Eq8.3.23}
g \cdot (e_{1} \wedge \ldots \wedge e_{k}) = \Delta_{I_{0,k}}^{(k)}(g) e_{1} \wedge \ldots \wedge e_{k} +  \displaystyle \sum_{I \neq I_{0,k}}\Delta_{I}^{(k)}(g) e_{i_{1}} \wedge \ldots \wedge e_{i_{k}}
\end{equation}
for $e_{i_{1}} \wedge \ldots \wedge e_{i_{k}} \in V(\omega_{\alpha_{k}}) = \bigwedge^{k}(\mathbb{C}^{n+1})$ and $I = \{i_{1} < \cdots < i_{k}\} \subset \{1,\cdots,n+1\}$.

By applying Theorem \ref{C8S8.3Sub8.3.2Teo8.3.5} we obtain the following expression for the quasi-potential $\varphi \colon {\rm{SL}}(n+1,\mathbb{C}) \to \mathbb{R}$

$$\varphi(g) = \displaystyle  \frac{1}{2\pi} \log \Big (\prod_{k = 1}^{n} \big| \big|g(e_{1} \wedge \ldots \wedge e_{k})\big|\big|^{2\langle \delta_{B},h_{\alpha_{k}}^{\vee} \rangle} \Big).$$
Notice that we can also write
$$\varphi(g) = \displaystyle  \sum_{k = 1}^{n}\frac{\langle \delta_{B},h_{\alpha_{k}}^{\vee} \rangle}{2\pi} \log \Big ( \big| \big|g(e_{1} \wedge \ldots \wedge e_{k})\big|\big|^{2} \Big),$$
remember that $\delta_{B} = \sum_{\alpha \in \Pi^{+}} \alpha$. Now by taking a local section $s_{U} \colon U \to {\rm{SL}}(n+1,\mathbb{C})$, where $U = N^{-}B \subset X_{B}$, defined by 
\begin{center}

$s_{U}(n^{-}(z)B) = n^{-}(z)  \in {\rm{SL}}(n+1,\mathbb{C})$, \ \ $\forall n^{-}(z)B \in U$, 

\end{center}
we obtain the following local expression

\begin{center}

$\omega_{X_{B}}|_{U} = \sqrt{-1} \partial \overline{\partial}(s_{U}^{\ast}(\varphi))$.

\end{center}
Hence, we need to determine $s_{U}^{\ast}(\varphi)$ in order to compute the local expression of $\omega_{X_{B}}$ and $\omega_{CY}$ . By definition of $s_{U}$ we have

\begin{center}

$s_{U}^{\ast}(\varphi)(n^{-}(z)B) = \varphi(s_{U}(n^{-}(z)B)) = \varphi(n^{-}(z))$. 

\end{center}
From equation \ref{C8S8.3Sub8.3.2Eq8.3.23} and the functions $\Delta_{I}^{(k)} \colon {\rm{SL}}(n+1,\mathbb{C}) \to \mathbb{C}$ we obtain
$$\varphi(n^{-}(z)) = \displaystyle  \sum_{k = 1}^{n}\frac{\langle \delta_{B},h_{\alpha_{k}}^{\vee} \rangle}{2\pi} \log \Big (1 + \sum_{I \neq I_{0,k}} \big|\Delta_{I}^{(k)}(n^{-}(z)) \big|^{2}\Big),$$
therefore we have 
$$\omega_{X_{B}}|_{U} = \displaystyle  \sum_{k = 1}^{n}\frac{\langle \delta_{B},h_{\alpha_{k}}^{\vee} \rangle}{2\pi} \sqrt{-1}\partial \overline{\partial}\log \Big (1 + \sum_{I \neq I_{0,k}} \big|\Delta_{I}^{(k)}(n^{-}(z)) \big|^{2}\Big).$$

Now we consider local coordinates $(n(z)B,\xi) \in K_{X_{B}}|_{U}$. On these coordinates we can write 
$$
\nabla \xi|_{U} = d\xi + 2\pi \xi \partial (s_{U}^{\ast}(\varphi)),
$$
thus we obtain
$$\nabla \xi|_{U} = d\xi + \displaystyle  \sum_{k = 1}^{n}\langle \delta_{B},h_{\alpha_{k}}^{\vee} \rangle \xi \partial \log \Bigg (1 + \sum_{I \neq I_{0,k}}\big |\Delta_{I}^{(k)}(n^{-}(z)) \big|^{2}\Bigg).$$

From Theorem \ref{C8S8.3Sub8.3.2Teo8.3.5} we have $\omega_{CY}$ over the open set $K_{X_{B}}|_{U}$ completelly determined by the forms

\begin{itemize}

    \item $\omega_{X_{B}} |_{U} = \displaystyle  \sum_{k = 1}^{n}\frac{\langle \delta_{B},h_{\alpha_{k}}^{\vee} \rangle}{2\pi} \sqrt{-1} \partial \overline{\partial}\log \Big (1 + \sum_{I \neq I_{0,k}} \big|\Delta_{I}^{(k)}(n^{-}(z)) \big|^{2}\Big)$; \ ( Horizontal )
    
    \item $\nabla \xi |_{U} = d\xi + \displaystyle  \sum_{k = 1}^{n}\langle \delta_{B},h_{\alpha_{k}}^{\vee} \rangle \xi \partial \log \Big (1 + \sum_{I \neq I_{0,k}} \big|\Delta_{I}^{(k)}(n^{-}(z)) \big|^{2}\Big)$. \ (Vertical)
    
\end{itemize}
It is worth to notice that we can compute $\langle \delta_{B},h_{\alpha_{k}}^{\vee} \rangle$, $1 \leq k \leq n$, by means of the Cartan matrix of $\mathfrak{sl}(n+1,\mathbb{C})$.\end{proof}

\subsubsection{Flags of symplectic group: full flag of ${\rm{Sp}}(4,\mathbb{C})$}

Consider  $G^{\mathbb{C}} = {\rm{Sp}}(4,\mathbb{C})$ and fix a Cartan subalgebra of $\mathfrak{sp}(4,\mathbb{C})$ given by diagonal matrices whose the trace is equal zero. The simple root system in this case is given by

\begin{center}

$\Sigma  = \Big \{\alpha_{1} = \epsilon_{1} - \epsilon_{2}, \alpha_{2} = 2\epsilon_{2} \Big \}$.

\end{center}
By choosing  $\Theta = \emptyset$ we have $P_{\Theta} = B$, and the corresponding full flag manifold $X_{B} = {\rm{Sp}}(4,\mathbb{C})/B$. Notice that if we consider ${\rm{Sp}}(2) \subset {\rm{Sp}}(4,\mathbb{C})$ the compact real form of  ${\rm{Sp}}(4,\mathbb{C})$ we have 
$$X_{B} = {\rm{Sp}}(4,\mathbb{C})/B \cong {\rm{Sp}}(2)/T^{2},$$
where $T^{2} \subset {\rm{Sp}}(2)$ denotes a maximal torus of ${\rm{Sp}}(2)$.

In order to get the basic data to describe the metric which we are interested in, we need to understand the subgroup $N^{-} = \exp(\mathfrak{n}^{-})$, since we will work on the open set $U = N^{-}x_{0}$, where $x_{0} = eB \in X_{B}$.

In this case, $\mathfrak{n}^{-} \subset \mathfrak{sp}(4,\mathbb{C})$ is given by matrices of the form 

\begin{center}

$Z = \begin{pmatrix} 
0 & 0 & 0 & 0\\
z_{1} & 0 & 0 & 0\\
z_{2} & z_{3} & 0 & -z_{1}\\
z_{3} & z_{4} & 0 & 0
\end{pmatrix}$.

\end{center}
Since $N^{-} = \exp(\mathfrak{n}^{-})$, let us compute the exponential of an arbitrary matrix $Z \in \mathfrak{n}^{-}$. A straightforward computation provides

\begin{center}

$Z^{2} = \begin{pmatrix} 
0 & 0 & 0 & 0\\
0 & 0 & 0 & 0\\
0 & -z_{1}z_{4} & 0 & 0\\
z_{1}z_{4} & 0 & 0 & 0
\end{pmatrix}$, \ \ $Z^{3} = \begin{pmatrix} 
0 & 0 & 0 & 0\\
0 & 0 & 0 & 0\\
-z_{1}^{2}z_{4} & 0& 0 & 0\\
0 & 0 & 0 & 0
\end{pmatrix}$,
\end{center}
and  $Z^{4} = 0$. Therefore, since for all $Z \in \mathfrak{n}^{-}$ we have  $\exp(Z)$ given by
$$\exp(Z) = \mathbb{1}_{4} + Z + \frac{1}{2!}Z^{2} + \frac{1}{3!}Z^{3},$$
we can write
$$\exp(Z) = \begin{pmatrix} 
1 & 0 & 0 & 0\\
z_{1} & 1 & 0 & 0\\
p_{1}(z) & p_{3}(z) & 1 & -z_{1}\\
p_{2}(z) & z_{4} & 0 & 1
\end{pmatrix},$$
where $z = (z_{1},z_{2},z_{3},z_{4}) \in \mathbb{C}^{4}$, and $p_{j}(z) \in \mathbb{C}[z]$, $j = 1,2,3$, are polynomials given by
\begin{center}
$p_{1}(z) = z_{2} - \frac{1}{3!}z_{1}^{2}z_{4}, \ \ p_{2}(z) = z_{3} + \frac{1}{2!}z_{1}z_{4}, \ \ p_{3}(z) = z_{3} - \frac{1}{2!}z_{1}z_{4}$.
\end{center}
\begin{remark}
It is worth to observe that different from the previous cases in this case the coordinate system provided by the exponential map has a number polynomials greater than the dimension of the big cell, see Remark \ref{coordinates}. Thus we will work with the above description of coordinate functions.
\end{remark}
By using the above computations we parameterize the set $U = N^{-}x_{0}$ in the following way: 

\begin{center}
$z= (z_{1},z_{2},z_{3},z_{4}) \in \mathbb{C}^{4} \mapsto \exp(Z)x_{0} = \begin{pmatrix} 
1 & 0 & 0 & 0\\
z_{1} & 1 & 0 & 0\\
p_{1}(z) & p_{3}(z) & 1 & -z_{1}\\
p_{2}(z) & z_{4} & 0 & 1
\end{pmatrix}x_{0}$.
\end{center}

Other important element to consider in our description are the fundamental representations. In this case we have

\begin{center}
$V(\omega_{\alpha_{1}}) = \mathbb{C}^{4}$ \ \ and \ \ $V(\omega_{\alpha_{2}}) \subset \bigwedge^{2}(\mathbb{C}^{4})$,
\end{center}
where $v_{\omega_{\alpha_{1}}}^{+} = e_{1}$ and $v_{\omega_{\alpha_{2}}}^{+} = e_{1} \wedge e_{2}$. We remark that, different from the case $\mathfrak{sl}(n,\mathbb{C})$, the space $V(\omega_{\alpha_{2}}) \subset \bigwedge^{2}(\mathbb{C}^{4})$ is not the whole space $\bigwedge^{2}(\mathbb{C}^{4})$. Instead, we have a maximal invariant subspace defined by $V(\omega_{\alpha_{2}}) = U(\mathfrak{sp}(4,\mathbb{C}))v_{\alpha_{2}}^{+}$, where $U(\mathfrak{sp}(4,\mathbb{C}))$ defines the enveloping algebra associated to $\mathfrak{sp}(4,\mathbb{C})$.

The potential $\varphi \colon {\rm{Sp}}(4,\mathbb{C}) \to \mathbb{R}$ for $\omega_{X_{B}} \in \Omega^{(1,1)}(X_{B})^{{\rm{Sp}}(2)}$ in this case is given by 
$$\varphi(g) = \displaystyle \frac{\langle \delta_{B},h_{\alpha_{1}}^{\vee} \rangle}{2\pi}\log \bigg ( \big|\big|ge_{1}\big|\big|^{2} \bigg ) + \displaystyle \frac{\langle \delta_{B},h_{\alpha_{2}}^{\vee} \rangle}{2\pi}\log \bigg ( \big|\big|g(e_{1} \wedge e_{2})\big|\big|^{2} \bigg ).
$$
We denote by $n^{-}(z) = \exp(Z) \in N^{-}$, and consider the local section $s_{U} \colon U \subset X_{B} \to {\rm{Sp}}(4,\mathbb{C})$ such that $s_{U}(n^{-}(z)x_{0}) = n^{-}(z)$. From these we have $\omega_{X_{B}}|_{U} = \sqrt{-1}\partial \overline{\partial}(s_{U}^{\ast}(\varphi))$ and
$$\varphi(s_{U}(n^{-}(z)x_{0})) = \varphi(n^{-}(z)).$$

In order to compute the above expression, let us introduce the following notation: given $1 \leq i < j \leq 4$ we define   
\begin{equation}\label{def-matrix}
\displaystyle \det_{ij} \begin{pmatrix} 
1 & 0 \\
z_{1} & 1\\
p_{1}(z) & p_{3}(z)\\
p_{2}(z) & z_{4}
\end{pmatrix},
\end{equation}
as being the determinant of the submatrix $2 \times 2$ determined by the {\it row} $i$ and the {\it row} $j$ of the matrix involved in the expression \ref{def-matrix}. This last convention allows us to write 

$$\varphi(n^{-}(z)) = \displaystyle \frac{\langle \delta_{B},h_{\alpha_{1}}^{\vee} \rangle}{2\pi}\log \Big ( 1+|z_{1}|^{2} + \sum_{j = 1}^{2}\big |p_{j}(z) \big|^{2}\Big )+\displaystyle \frac{\langle \delta_{B},h_{\alpha_{2}}^{\vee} \rangle}{2\pi}\log \Bigg (\sum_{1\leq i<j\leq 4} \Bigg | \det_{ij} \begin{pmatrix} 
1 & 0 \\
z_{1} & 1\\
p_{1}(z) & p_{3}(z)\\
p_{2}(z) & z_{4}
\end{pmatrix} \Bigg |^{2} \Bigg ).$$ 

Since the Cartan matrix $\mathscr{C} = (C_{ij})$ of ${\mathfrak{sp}}(4,\mathbb{C})$ is given by 
\begin{center}
$\mathscr{C} = \begin{pmatrix} 
2 & -1 \\
-2 & \ 2
\end{pmatrix}$, \ \ with \ \ $C_{ij} = \displaystyle \frac{2\kappa(\alpha_{i},\alpha_{j})}{\kappa(\alpha_{j},\alpha_{j})}$,

\end{center}
by considering the set of positive roots $\Pi^{+} = \big \{ \alpha_{1},\alpha_{2},\alpha_{1} +  \alpha_{2}, 2\alpha_{1} + \alpha_{2}\big \}$, we have  $\delta_{B} = 4\alpha_{1} + 3 \alpha_{2}$. Therefore we obtain

\begin{center}
$\langle \delta_{B},h_{\alpha_{1}}^{\vee} \rangle = 4C_{11} + 3C_{21} = 2$ \ \ e \ \ $\langle \delta_{B},h_{\alpha_{2}}^{\vee} \rangle = 4C_{12} + 3C_{22} = 2$,
\end{center}
and
$$\varphi(n^{-}(z)) = \displaystyle \frac{1}{\pi} \Bigg [\log \Big ( 1+|z_{1}|^{2} + \sum_{j = 1}^{2}\big |p_{j}(z) \big |^{2}\Big ) + \log \Bigg (\sum_{1\leq i<j\leq 4} \Bigg | \det_{ij} \begin{pmatrix} 
1 & 0 \\
z_{1} & 1\\
p_{1}(z) & p_{3}(z)\\
p_{2}(z) & z_{4}
\end{pmatrix} \Bigg |^{2} \Bigg ) \Bigg ].$$ 
From this last expression we can determine $\omega_{X_{B}} |_{U} =  \sqrt{-1}\partial \overline{\partial}(s_{U}^{\ast}(\varphi))$ and $\nabla \xi |_{U} = \displaystyle d\xi +  2\pi\xi \partial(s_{U}^{\ast}(\varphi))$, where $U = N^{-}x_{0} \subset X_{B}$. By gathering the above ideas together we have just proved the following result
\begin{proposition}
Consider $G^{\mathbb{C}} = {\rm{Sp}}(4,\mathbb{C})$ and  $X_{B} = {\rm{Sp}}(4,\mathbb{C})/B \cong {\rm{Sp}}(2)/T^{2}$ the corresponding full flag manifold. Then the manifold defined by the total space $K_{X_{B}}$ admits a complete Ricci-flat K\"{a}hler metric $\omega_{CY}$ (locally) described by
$$\omega_{CY} = (2\pi |\xi|^{2} + C)^{\frac{1}{5}} \omega_{X_{B}} - \frac{\sqrt{-1}}{5}(2\pi |\xi|^{2} + C)^{-\frac{4}{5}}\nabla \xi \wedge \overline{\nabla \xi},$$
for some positive constant $C>0$, such that
$$\omega_{X_{B}} = \displaystyle \frac{\sqrt{-1}}{\pi} \partial \overline{\partial} \Bigg [\log \Big ( 1+|z_{1}|^{2} + \sum_{j = 1}^{2}\big |p_{j}(z) \big |^{2}\Big ) + \log \Bigg ( \sum_{1\leq i<j\leq 4} \Bigg | \det_{ij} \begin{pmatrix} 
1 & 0 \\
z_{1} & 1\\
p_{1}(z) & p_{3}(z)\\
p_{2}(z) & z_{4}
\end{pmatrix} \Bigg |^{2} \Bigg ) \Bigg ],$$
and
$$\nabla \xi = \displaystyle d\xi + 2\xi \partial \Bigg [\log \Big ( 1+|z_{1}|^{2} + \sum_{j = 1}^{2}\big |p_{j}(z) \big |^{2}\Big ) + \log \Bigg ( \sum_{1\leq i<j\leq 4} \Bigg | \det_{ij} \begin{pmatrix} 
1 & 0 \\
z_{1} & 1\\
p_{1}(z) & p_{3}(z)\\
p_{2}(z) & z_{4}
\end{pmatrix} \Bigg |^{2} \Bigg ) \Bigg ].$$
\end{proposition}

\subsubsection{Flags of special orthogonal group: an example of ${\rm{SO}}(8,\mathbb{C})$}
Consider $G^{\mathbb{C}} = {\rm{SO}}(8,\mathbb{C})$ and fix the Cartan subalgebra of $\mathfrak{so}(8,\mathbb{C})$ given by the diagonal matrices whose the trace is equal zero. The simple root system in this case is given by
$$\Sigma = \Big \{ \alpha_{1} = \epsilon_{1} - \epsilon_{2}, \alpha_{2} = \epsilon_{2} - \epsilon_{3}, \alpha_{3} = \epsilon_{3}-\epsilon_{4}, \alpha_{4} = \epsilon_{3}+\epsilon_{4} \Big \}.$$

By choosing $\Theta = \{\alpha_{3},\alpha_{4}\}$ we have the associated parabolic subgroup $P = P_{\Theta}$ which determines the complex flag manifold $X_{P} = {\rm{SO}}(8,\mathbb{C})/P$. If we consider the compact real form ${\rm{SO}}(8) \subset {\rm{SO}}(8,\mathbb{C})$ of ${\rm{SO}}(8,\mathbb{C})$, we also can write
$$X_{P} = {\rm{SO}}(8,\mathbb{C})/P \cong {\rm{SO}}(8)/ {\rm{SU}}(3) \times T^{2}$$
where ${\rm{SU}}(3) \times {\rm{U}}(1) \times {\rm{U}}(1) = P \cap {\rm{SO}}(8)$.

In order to describe the metric which we are interested in, we need to understand the subgroup $R_{u}(P)^{-} \subset N^{-}$. We start by remarking that the Lie algebra of  $R_{u}(P)^{-}$ is given by matrices of the form
\begin{equation}\label{mat-so-n}
Z = \begin{pmatrix} 
0     &   0    &    0     &   0    &  0  &   0    &   0    &   0  \\
z_{1} &   0    &    0     &   0    &  0  &   0    &   0    &   0  \\
z_{2} & z_{7}  &    0     &   0    &  0  &   0    &   0    &   0  \\
z_{3} & z_{8}  &    0     &   0    &  0  &   0    &   0    &   0  \\
0     & -z_{4} &  -z_{5}  & -z_{6} &  0  & -z_{1} & -z_{2} & -z_{3}\\
z_{4} &   0    &  -z_{9}  & -z_{10}&  0  &   0    & -z_{7} & -z_{8}  \\
z_{5} & z_{9}  &    0     &   0    &  0  &   0    &   0    &   0  \\
z_{6} & z_{10} &    0     &   0    &  0  &   0    &   0    &   0  
\end{pmatrix}.
\end{equation}

Since $R_{u}(P)^{-} = \exp({\text{Lie}}(R_{u}(P)^{-}))$ in order to parameterize the open cell $U = R_{u}(P)^{-}x_{0} \subset X_{P}$, we need to compute the exponential of the matrices of the form \ref{mat-so-n}. Given $Z \in {\text{Lie}}(R_{u}(P)^{-})$ we have \begin{center}
$\exp(Z) = \mathbb{1}_{8} + Z + \frac{1}{2!}Z^{2} + \frac{1}{3!}Z^{3} + \frac{1}{4!}Z^{4}$,
\end{center}
that is, 
$$\exp(Z) = \begin{pmatrix} 
1        &   0      &   0     &   0      &  0  &   0    &   0       &   0  \\
z_{1}    &   1      &   0     &   0      &  0  &   0    &   0       &   0  \\
p_{1}(z) & z_{7}    &   1     &   0      &  0  &   0    &   0       &   0  \\
p_{2}(z) & z_{8}    &   0     &   1      &  0  &   0    &   0       &   0  \\
p_{3}(z) & p_{7}(z) & p_{9}(z)& p_{10}(z)&  1  & -z_{1} & p_{11}(z) &  p_{12}(z)\\
p_{4}(z) & p_{8}(z) & -z_{9}  & -z_{10}  &  0  &   1    & -z_{7}    & -z_{8}  \\
p_{5}(z) & z_{9}    &   0     &   0      &  0  &   0    &   1       &   0  \\
p_{6}(z) & z_{10}   &   0     &   0      &  0  &   0    &   0       &   1  
\end{pmatrix},$$
where $z = (z_{1},\ldots,z_{10}) \in \mathbb{C}^{10}$, and $p_{1}(z),\ldots,p_{12}(z) \in \mathbb{C}[z]$ are given by
$$p_{1}(z) = z_{2} + \frac{1}{2!}z_{1}z_{7}, \ \ p_{2}(z) = z_{3} + \frac{1}{2!}z_{1}z_{8}, \ \ p_{3}(z) = -\frac{2}{2!}(z_{1}z_{4} + z_{2}z_{5} +z_{3}z_{6}) + \frac{2}{4!}z_{1}(z_{1}z_{7}z_{9}+z_{1}z_{8}z_{10}),$$
$$p_{4}(z) = z_{4} - \frac{1}{2!} (z_{2}z_{9} + z_{7}z_{5}+z_{3}z_{10}+z_{8}z_{6}) - \frac{2}{3!}z_{1}(z_{7}z_{9}+z_{8}z_{10}), \ \ p_{5}(z) = z_{5} + \frac{1}{2!}z_{1}z_{9}, \ \ p_{6}(z) = z_{6} + \frac{1}{2!}z_{1}z_{10},$$
$$p_{7}(z) = -z_{4} - \frac{1}{2!}(z_{2}z_{9} + z_{7}z_{5} + z_{3}z_{10} + z_{8}z_{6}) + \frac{2}{3!}(z_{1}z_{7}z_{9} + z_{1}z_{8}z_{10}), \ \ p_{8}(z) = -\frac{2}{2!}(z_{7}z_{9} + z_{8}z_{10}),$$
$$p_{9}(z) = -z_{5} + \frac{2}{2!}z_{1}z_{9}, \ \ p_{10}(z) = -z_{6} + \frac{2}{2!}z_{1}z_{10}, \ \ p_{11}(z) = -z_{2} + \frac{2}{2!}z_{1}z_{7}, \ \ p_{12}(z) = -z_{3} +  \frac{2}{2!}z_{1}z_{8}.$$
Therefore the parameterization of the open cell $U = R_{u}(P)^{-}x_{0}$ is given by
$$z = (z_{1},\ldots,z_{10}) \in \mathbb{C}^{10} \mapsto \exp(Z)x_{0} \in  R_{u}(P)^{-}x_{0}.$$
The fundamental representations to consider in this case are given by
$$
V(\omega_{\alpha_{1}}) = \mathbb{C}^{8} \mbox{  and  } V(\omega_{\alpha_{2}}) = \bigwedge^{2}(\mathbb{C}^{8}),
$$
with $v_{\omega_{\alpha_{1}}}^{+} = e_{1}$ and $v_{\omega_{\alpha_{2}}}^{+} = e_{1} \wedge e_{2}$.\\ 

The potential $\varphi \colon {\rm{SO}}(8,\mathbb{C}) \to \mathbb{R}$ for $\omega_{X_{P}} \in \Omega^{(1,1)}(X_{P})^{{\rm{SO}}(8)}$ is defined by
$$\varphi(g) = \displaystyle \frac{\langle \delta_{P},h_{\alpha_{1}}^{\vee} \rangle}{2\pi}\log \big ( ||ge_{1}||^{2} \big ) + \displaystyle \frac{\langle \delta_{P},h_{\alpha_{2}}^{\vee} \rangle}{2\pi}\log \big ( ||g(e_{1} \wedge e_{2})||^{2} \big ).$$
By denoting $n^{-}(z) = \exp(Z) \in R_{u}(P)^{-}$, we can take the local section $s_{U} \colon U \subset X_{B} \to {\rm{Sp}}(4,\mathbb{C})$, such that $s_{U}(n^{-}(z)x_{0}) = n^{-}(z)$. Since $s_{U}^{\ast} \varphi = \varphi \circ s_{U}$, we obtain
$$\varphi(s_{U}(n^{-}(z)x_{0})) = \varphi(n^{-}(z)).$$

We will use a notation similar to the previous case, i.e., given $1 \leq i < j \leq 8$ we define
\begin{equation}\label{det-so-n}
\displaystyle \det_{ij} \begin{pmatrix} 
    1    & 0 \\
 z_{1}   & 1\\
p_{1}(z) & z_{7}\\
p_{2}(z) & z_{8}\\
p_{3}(z) & p_{7}(z)\\
p_{4}(z) & p_{8}(z)\\
p_{5}(z) & z_{9}\\
p_{6}(z) & z_{10} 
\end{pmatrix},
\end{equation}
as being the determinant of the $2\times 2$ submatrix defined by the row $i$ and by the row $j$ of the matrix $8 \times 2$ described by \ref{det-so-n}. From this we can write 
$$\varphi(n^{-}(z)) = \displaystyle \frac{\langle \delta_{P},h_{\alpha_{1}}^{\vee} \rangle}{2\pi}\log \Big ( 1+|z_{1}|^{2} + \sum_{j = 1}^{6}\big |p_{j}(z) \big|^{2}\Big )+\displaystyle \frac{\langle \delta_{P},h_{\alpha_{2}}^{\vee} \rangle}{2\pi}\log \Bigg (\sum_{1\leq i<j\leq 8} \Bigg | \det_{ij} \begin{pmatrix} 
    1    & 0 \\
 z_{1}   & 1\\
p_{1}(z) & z_{7}\\
p_{2}(z) & z_{8}\\
p_{3}(z) & p_{7}(z)\\
p_{4}(z) & p_{8}(z)\\
p_{5}(z) & z_{9}\\
p_{6}(z) & z_{10} 
\end{pmatrix} \Bigg |^{2} \Bigg ).$$ 

In order to compute the coefficients $\langle \delta_{P},h_{\alpha_{1}}^{\vee} \rangle$ and $\langle \delta_{P},h_{\alpha_{2}}^{\vee} \rangle$ we use the Cartan matrix of $\mathfrak{so}(8,\mathbb{C})$. Let $\mathscr{C} = (C_{ij})$ be the Cartan matrix of $\mathfrak{so}(8,\mathbb{C})$, in this case we have
\begin{center}
$\mathscr{C} = \begin{pmatrix} 
\ 2 & -1 & \ \ 0 & \ \ 0\\
-1 & \ \ 2 & -1 & -1\\
\ 0 & -1 &  \ \ 2 & \ \ 0\\
\ 0 & -1 & \ \ 0 & \ \ 2
\end{pmatrix}$, \ \ with \ \ $C_{ij} = \displaystyle \frac{2\kappa(\alpha_{i},\alpha_{j})}{\kappa(\alpha_{j},\alpha_{j})}$.
\end{center}

Since $\Theta = \{\alpha_{3},\alpha_{4}\}$ it follows that
$$\Pi^{+} \backslash \langle \Theta \rangle^{+}= 
  \Bigg \{ \begin{matrix}
  \alpha_{1}, \ \ \alpha_{2}, \ \ \alpha_{1} + \alpha_{2}, \ \ \alpha_{2} + \alpha_{3}, \ \ \alpha_{2} + \alpha_{4}, \ \ \alpha_{1} + \alpha_{2} + \alpha_{3}, \ \ \alpha_{1} + \alpha_{2} + \alpha_{4},  \\
  \alpha_{2} + \alpha_{3} + \alpha_{4}, \ \ \alpha_{1} + \alpha_{2} + \alpha_{3} + \alpha_{4}, \ \ \alpha_{1} + 2 \alpha_{2} + \alpha_{3} + \alpha_{4}
 \end{matrix} \Bigg \}.
 $$
By a direct computation we obtain 
$$\delta_{P} = 6\alpha_{1} + 10 \alpha_{2} + 5 \alpha_{3} + 5 \alpha_{4},$$
and 
\begin{center}
$\langle \delta_{P},h_{\alpha_{1}}^{\vee} \rangle = 6 C_{11} + 10 C_{21} + 5 C_{31} + 5 C_{41} = 2, \ \ {\text{and}} \ \ \langle \delta_{P},h_{\alpha_{2}}^{\vee} \rangle = 6 C_{12} + 10 C_{22} + 5 C_{32} + 5 C_{42} = 4$, 

\end{center}
therefore we have 
$$\varphi(n^{-}(z)) = \displaystyle  \frac{1}{\pi} \Bigg [  \log \Big ( 1+|z_{1}|^{2} + \sum_{j = 1}^{6}\big |p_{j}(z) \big|^{2}\Big ) + 2\log \Bigg (\sum_{1\leq i<j\leq 8} \Bigg | \det_{ij} \begin{pmatrix} 
    1    & 0 \\
 z_{1}   & 1\\
p_{1}(z) & z_{7}\\
p_{2}(z) & z_{8}\\
p_{3}(z) & p_{7}(z)\\
p_{4}(z) & p_{8}(z)\\
p_{5}(z) & z_{9}\\
p_{6}(z) & z_{10} 
\end{pmatrix} \Bigg |^{2} \Bigg ) \Bigg ].$$
By Gathering the above ideas together we can describe
\begin{center}
$\omega_{X_{P}} |_{U} =  \sqrt{-1}\partial \overline{\partial}(s_{U}^{\ast}(\varphi))$ \ \ and \ \ $\nabla \xi |_{U} = \displaystyle d\xi +  2\pi\xi \partial(s_{U}^{\ast}(\varphi))$,
\end{center}
where $U = R_{u}(P)^{-}x_{0} \subset X_{P}$. From these we have shown the following result

\begin{proposition}
Let $G^{\mathbb{C}} = {\rm{SO}}(8,\mathbb{C})$ and $X_{P} = {\rm{SO}}(8,\mathbb{C})/P \cong {\rm{SO}}(8)/ {\rm{SU}}(3) \times T^{2}$, where $P = P_{\{\alpha_{3},\alpha_{4}\}}$. Then the total space of $K_{X_{P}}$ admits a complete Ricci-flat K\"{a}hler metric $\omega_{CY}$ (locally) described by

\begin{center}
$\omega_{CY} = (2\pi |\xi|^{2} + C)^{\frac{1}{11}} \omega_{X_{P}} - \frac{\sqrt{-1}}{11}(2\pi |\xi|^{2} + C)^{-\frac{10}{11}}\nabla \xi \wedge \overline{\nabla \xi}$,
\end{center}
for some positive constant $C>0$, such that 
$$\omega_{X_{B}} = \displaystyle \frac{\sqrt{-1}}{\pi}\partial \overline{\partial} \Bigg [ \log \Big ( 1+|z_{1}|^{2} + \sum_{j = 1}^{6}\big |p_{j}(z) \big|^{2}\Big ) + 2\log \Bigg (\sum_{1\leq i<j\leq 8} \Bigg | \det_{ij} \begin{pmatrix} 
    1    & 0 \\
 z_{1}   & 1\\
p_{1}(z) & z_{7}\\
p_{2}(z) & z_{8}\\
p_{3}(z) & p_{7}(z)\\
p_{4}(z) & p_{8}(z)\\
p_{5}(z) & z_{9}\\
p_{6}(z) & z_{10} 
\end{pmatrix} \Bigg |^{2} \Bigg ) \Bigg ],$$ 
and
$$\nabla \xi = \displaystyle d\xi +   2\xi \partial \Bigg [\log \Big ( 1+|z_{1}|^{2} + \sum_{j = 1}^{6}\big |p_{j}(z) \big|^{2}\Big ) +  2 \log \Bigg (\sum_{1\leq i<j\leq 8} \Bigg | \det_{ij} \begin{pmatrix} 
    1    & 0 \\
 z_{1}   & 1\\
p_{1}(z) & z_{7}\\
p_{2}(z) & z_{8}\\
p_{3}(z) & p_{7}(z)\\
p_{4}(z) & p_{8}(z)\\
p_{5}(z) & z_{9}\\
p_{6}(z) & z_{10} 
\end{pmatrix} \Bigg |^{2} \Bigg ) \Bigg ].$$
\end{proposition}

\begin{remark}
It is worth to point out that the above description of Calabi's metrics on the canonical bundle of complex flag manifolds can be reproduced for other complex flag manifolds associated to the classical groups $G^{\mathbb{C}} = {\rm{SL}}(n,\mathbb{C}), {\rm{SO}}(n,\mathbb{C})$ and ${\rm{Sp}}(2n,\mathbb{C})$.


\end{remark}

\section{Final remarks: flag manifolds associated to exceptional Lie groups}
We finish our discussion with a basic description of how the result of Theorem \ref{C8S8.3Sub8.3.2Teo8.3.5} can be applied in some complex flag manifolds on which the underlying complex simple Lie algebra is not of the classical type.


\subsection{The complex Cayley plane} \label{sec-cayley}
Consider the case $G^{\mathbb{C}} = E_{6}$. Here our conventions for the Lie algebra structure are according to \cite{SMA} and our approach of the complex Cayley plane is according to \cite{CAYLEY}.  
\begin{figure}[H]
\centering
\includegraphics[scale=0.6]{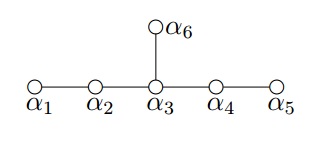}
\caption{Dynkin diagram associated to the Lie algebra $E_{6}$.}
\end{figure}
The complex Cayley plane is the complex flag manifold $\mathbb{O}{\rm{P}}^{2}$ obtained from $\Theta = \Sigma \backslash \{\alpha_{5}\}$, namely 

\begin{center}

$\mathbb{O}{\rm{P}}^{2} = E_{6}/P_{\Theta}$.

\end{center}
For this manifold we have $\dim_{\mathbb{C}}(\mathbb{O}\rm{P}^{2}) = 16$. From Propositions \ref{C8S8.2Sub8.2.3P8.2.6} and \ref{C8S8.2Sub8.2.3P8.2.6} we have
$${\text{Pic}}(\mathbb{O}{\rm{P}}^{2}) = \mathbb{Z}[\eta_{\alpha_{5}}].$$

From Proposition \ref{C8S8.2Sub8.2.3Eq8.2.35} we obtain
$$-K_{\mathbb{O}{\rm{P}}^{2}} = L_{\chi_{\omega_{\alpha_{5}}}}^{\otimes \langle \delta_{P_{\Theta}},h_{\alpha_{5}}^{\vee} \rangle}.$$

The K\"{a}hler form $\omega_{\mathbb{O}{\rm{P}}^{2}} \in c_{1}(\mathbb{O}{\rm{P}}^{2})$ can be described by means of the quasi-potential $\varphi \colon E_{6} \to \mathbb{R}$ given by
$$\varphi(g) = \displaystyle \frac{\langle \delta_{P_{\Theta}}, h_{\alpha_{5}}^{\vee}\rangle}{2\pi} \log ||g v_{\omega_{\alpha_{5}}}^{+}||^{2}.$$
As in the previous examples we have (locally) $\omega_{\mathbb{O}{\rm{P}}^{2}}|_{U} = \sqrt{-1}\partial \overline{\partial}(s_{U}^{\ast}(\varphi))$, for some local section $s_{U} \colon U \subset \mathbb{O}{\rm{P}}^{2} \to E_{6}$. Therefore, we can apply Theorem \ref{C8S8.3Sub8.3.2Teo8.3.5} and obtain a complete Ricci-flat metric on $K_{\mathbb{O}{\rm{P}}^{2}}$ with associated K\"{a}hler form $\omega_{CY}$ described (locally) by the expression
$$\omega_{CY} = (2\pi |\xi|^{2} + C)^{\frac{1}{17}} \omega_{\mathbb{O}{\rm{P}}^{2}} - \frac{\sqrt{-1}}{17}(2\pi |\xi|^{2} + C)^{-\frac{16}{17}}\nabla \xi \wedge \overline{\nabla \xi},$$
for some positive constant $C>0$. Thus $(K_{\mathbb{O}{\rm{P}}^{2}},\omega_{CY})$ defines a complete noncompact Calabi-Yau manifold with complex dimension $17$. The explicit local expression of $\omega_{CY}$ in this case is non-trivial since the matrix realization of $E_{6}$ is quite complicated.

\subsection{Freudenthal variety}
Now consider $G^{\mathbb{C}} = E_{7}$. As before our conventions for the Lie algebra structure are according to \cite{SMA}. 
\begin{figure}[H]
\centering
\includegraphics[scale=0.5]{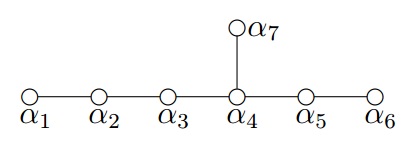}
\caption{Dynkin diagram associated to the Lie algebra $E_{7}$.}
\end{figure}
The Freudenthal variety is defined by the complex $27$-dimensional flag manifold associated to $\Theta = \Sigma \backslash \{\alpha_{6}\}$, namely $X_{P_{\Theta}} = E_{7}/P_{\Theta}$. From Propositions \ref{C8S8.2Sub8.2.3P8.2.6} and \ref{C8S8.2Sub8.2.3P8.2.6} we have
$$
{\text{Pic}}(E_{7}/P_{\Theta}) = \mathbb{Z}[\eta_{\alpha_{6}}].$$
Thus we obtain $-K_{E_{7}/P_{\Theta}} = L_{\chi_{\omega_{\alpha_{6}}}}^{\otimes \langle \delta_{P_{\Theta}}, h_{\alpha_{6}}^{\vee}\rangle}$. From these the quasi-potential $\varphi \colon E_{7} \to \mathbb{R}$ which defines $\omega_{E_{7}/P_{\Theta}} \in c_{1}(E_{7}/P_{\Theta})$ is given by
$$\varphi(g) = \displaystyle \frac{\langle \delta_{P_{\Theta}}, h_{\alpha_{6}}^{\vee}\rangle}{2\pi} \log ||g v_{\omega_{\alpha_{6}}}^{+}||^{2}.$$

Now if we apply Theorem \ref{C8S8.3Sub8.3.2Teo8.3.5} on $K_{E_{7}/P_{\Theta}}$ we obtain a complete Ricci-flat metric with associated K\"{a}hler form 
$$\omega_{CY} = (2\pi |\xi|^{2} + C)^{\frac{1}{28}} \omega_{E_{7}/P_{\Theta}} - \frac{\sqrt{-1}}{28}(2\pi |\xi|^{2} + C)^{-\frac{27}{28}}\nabla \xi \wedge \overline{\nabla \xi},$$
for some positive constant $C>0$. From these we obtain by Theorem \ref{C8S8.3Sub8.3.2Teo8.3.5} a complete noncompact Calabi-Yau manifold $(K_{E_{7}/P_{\Theta}},\omega_{CY})$. As in the previous case the explicit description of the local expression of the K\"{a}hler form $\omega_{CY}$ above is non-trivial since the underlying complex Lie algebra in this context has dimension $133$.

\begin{remark}
As we have seen in the last examples, for flag manifolds associated to exceptional complex simple Lie algebras the application of Theorem \ref{C8S8.3Sub8.3.2Teo8.3.5} becomes very complicated. The main reason is that for these types of Lie algebras we do not have a manageable matrix realization, thus we can not directly derive a suitable local expression for the potential which defines the first Chern class of $X_{P}$. 
\end{remark}

\end{document}